\documentclass[10pt,a4paper]{amsart}
\usepackage[utf8]{inputenc}
\usepackage[T1]{fontenc}
\usepackage{amsmath, amssymb,amsthm}
\usepackage[final]{ifdraft}      
\usepackage[dvipsnames]{xcolor}

\usepackage{tikz}
\usepackage{hyperref}
\hypersetup{
    unicode=true,          
    pdftitle={},    
    pdfauthor={},     
    pdfsubject={},   
    pdfkeywords={}{}, 
    pdfborder={0 0 0}
}

\usepackage{enumitem} 
\setenumerate[0]{label=(\roman*)}   

\numberwithin{equation}{section}
\newtheorem{thm}{Theorem}[section]

\newtheorem{prop}[thm]{Proposition}
\newtheorem{lemm}[thm]{Lemma}

\renewcommand{\hat}{\widehat}   
\renewcommand{\phi}{{\varphi}}

\newcommand{\N}{{\mathbb N}}
\newcommand{\Z}{{\mathbb Z}}
\newcommand{\R}{{\mathbb R}}
\let\C\relax 
\newcommand{\C}{{\mathbb C}}
\newcommand{\Q}{{\mathbb Q}}

\newcommand{\Fbb}{{\mathbb F}}
\newcommand{\Pbb}{{\mathbb P}}
\newcommand{\Tbb}{{\mathbb T}}
\newcommand{\Ebb}{{\mathbb E}}

\newcommand{\Acal}{{\mathcal A}}
\newcommand{\Ecal}{{\mathcal E}}
\newcommand{\Lcal}{{\mathcal L}}
\newcommand{\Pcal}{{\mathcal P}}
\newcommand{\Ncal}{{\mathcal N}}

\newcommand{\omegatuple}{{\underline{\omega}}}

\newcommand{\dd}{{\,\mathrm{d}}}

\DeclareMathOperator{\Mat}{Mat}

\DeclareMathOperator{\SL}{SL}
\DeclareMathOperator{\GL}{GL}
\DeclareMathOperator{\Supp}{Supp}
\DeclareMathOperator{\Span}{Span}
\DeclareMathOperator{\Stab}{Stab}
\DeclareMathOperator{\indic}{\mathbf{1}}
\DeclareMathOperator{\tr}{tr}           
\DeclareMathOperator{\CH}{CH} 

\newcommand{\abs}[1]{\lvert#1\rvert}    
\newcommand{\bigabs}[1]{\left\lvert{#1}\right\rvert}

\newcommand{\norm}[1]{\lVert#1\rVert}   
\newcommand{\bignorm}[1]{\bigl\lVert#1\bigr\rVert}   
\newcommand{\Bignorm}[1]{\Bigl\lVert#1\Bigr\rVert}   

\newcommand{\mybullet}{\,\cdot\,}

\newenvironment{newchanges}
{
   \ifoptionfinal{}{\color{red}}%
}%
{}
\newcommand{\newchange}[1]{{\ifoptionfinal{}{\color{red}} {#1}}}

\begin{document}
\title{Affine random walks on the torus}

\author{Weikun He}
\address{Einstein Institute of Mathematics, The Hebrew University of Jerusalem, Jerusalem 91904, Israel.}
\email{weikun.he@mail.huji.ac.il}

\author{Tsviqa Lakrec}
\email{tsviqa@gmail.com}

\author{Elon Lindenstrauss}
\email{elon@math.huji.ac.il}

\date{}

\begin{abstract}
We study quantitative equidistribution of random walks on the torus by affine transformations.
Under the assumption that the Zariski closure of the group generated by the linear part acts strongly irreducibly on $\R^d$ and is either Zariski connected or contains a proximal element, we give quantitative estimates (depending only on the linear part of the random walk) for how fast the random walk equidistributes unless the initial point and the translation part of the affine transformations can be perturbed so that the random walk is trapped in a finite orbit of small cardinality. In particular, we prove that the random walk equidistributes in law to the Haar measure if and only if the random walk is not trapped in a finite orbit.
\end{abstract}
\maketitle

\section{Introduction}

In this paper, we consider a random walk on the torus $\Tbb^d = \R^d/\Z^d$ for $d\geq 2$ using random elements from the group of affine transformations on this torus, and investigate under which condition on the initial point and the translation parts of the affine transformations this random walk equidistributes.

First let us recall what we know about the linear random walk. A quantitative equidistribution result for the linear random walk was proved by Bourgain, Furman, Mozes and the third named author in \cite{BFLM} and was extended by de Saxc\'e and the first named author in \cite{He_schubert, HS}. Qualitatively, these results imply the following (for which no purely ergodic theoretic proof is known):
\begin{thm}[\cite{BFLM}, \cite{HS}]
\label{thm:BFLM}
Let $\mu$ be a probability measure on $\SL_d(\Z)$ with a finite exponential moment, i.e. for some $\alpha>0$ we have that $\int \norm{g}^\alpha \dd \mu (g) <\infty$. Let $\Gamma$ denote the group generated by the support of $\mu$.
Assume that
\begin{equation}
\label{as:irreducible}
\text{the action of $\Gamma$ on $\R^d$ is strongly irreducible.}
\end{equation}
Assume also one of the following technical assumptions : 
\begin{equation}
\label{as:proximal}
\text{$\Gamma$ contains a proximal element,}
\end{equation}
or
\begin{equation}
\label{as:Zconnected}
\text{the Zariski closure of $\Gamma$ is connected.}
\end{equation}
Then for every starting point $x \in \Tbb^d$,
either $\mu^{*n} * \delta_x$ converges in the weak-$*$ topology to the normalized Haar measure on $\Tbb^d$ or $x$ is a periodic point for the random walk, i.e. the $\Gamma$-orbit of $x$ is finite.
\end{thm}
Recall that we say a group acts \emph{strongly irreducibly} on $\R^d$ if it does not preserve any nontrivial union of proper $\R$-linear subspaces of $\R^d$. A proximal element of $\SL_d(\R)$ is an element having a simple dominant eigenvalue. The word "connected" in \eqref{as:Zconnected} means connected for the Zariski topology (over $\C$).

The technical assumption~\eqref{as:proximal} is required in \cite{BFLM} and \eqref{as:Zconnected} is required in \cite{HS}; a slightly more technical condition, that is less restrictive than proximality that can be used instead of~\eqref{as:proximal} is given by the first named author in~\cite{He_schubert}.

Clearly, the two options in the conclusion are mutually exclusive.
Observe also that the $\Gamma$-orbit of $x$ is finite if and only if $x$ is rational, i.e. $x \in \Q^d/\Z^d$.

In this paper, we extend this result to affine random walks on $\Tbb^d$.
\begin{thm}
\label{thm:qualitative}
Let $\mu$ be a finitely supported probability measure on $\SL_d(\Z) \ltimes \Tbb^d$.
Let $H$ denote the group generated by the support of $\mu$ and let $\Gamma$ denote the projection of $H$ to $\SL_d(\Z)$.
Assume that $\Gamma$ satisfies \eqref{as:irreducible} and \newchange{either} \eqref{as:proximal} or \eqref{as:Zconnected}.
Then for every starting point $x \in \Tbb^d$, 
either $\mu^{*n} * \delta_x$ converges in the weak-$*$ topology to the normalized Haar measure on $\Tbb^d$ or \newchange{the random walk starting at $x$ is confined to a finite set (or equivalently the $H$-orbit of $x$ is finite)}.
\end{thm}

A special case was previously established by Boyer~\cite{Boyer}, where a Diophantine property of the coefficients of the translation parts is assumed\footnote{To be precise, unlike our result, Boyer does not assume that $\mu$ is finitely supported, \newchanges{but his result requires additional assumptions on the translation part that seem less natural when the support is not finite. See \cite[Example 1.3]{Boyer} for further discussion}.}. 
If we consider, instead of $\mu^{*n} * \delta_x$, the Ces\`aro mean  $\frac 1n\sum_{k=1}^n \mu^{*k} * \delta_x$, then the analogue of Theorem~\ref{thm:qualitative} for these means is a special case of a result of Benoist and Quint. 
Indeed, under the assumption~\ref{as:irreducible}, every $H$-invariant homogeneous probability measure on $\Tbb^d$ is either the Haar measure or a uniform counting measure on a finite $H$-orbit. By~\cite[Theorem 1.4(b)]{BQ3}, the measure  \newchange{$\frac{1}{n}\sum_{k=1}^n \mu^{*k} * \delta_x$} converges in the weak-$*$ topology to the $H$-invariant homogeneous measure supported on the closure of the orbit $Hx$. So Theorem~\ref{thm:qualitative} is new in that we have a convergence of $\mu^{*k} * \delta_x$ instead of the Ces\`aro mean.

Theorem~\ref{thm:qualitative} is a consequence of a quantitative equidistribution result, which (unless the random walk is very near a random walk on a small finite trajectory) has an equidistribution rate that depends only on the linear part of the random walk.
Let $(\Omega, \Pbb)$ be a probability space with a finite sample space $\Omega$ and such that $\Pbb(\omega) > 0$ for every $\omega \in \Omega$.
This non-degeneracy is assumed throughout this paper without mentioning.
Consider maps $\gamma \colon \Omega \to \SL_d(\Z)$ and $u \colon \Omega \to \Tbb^d$. Then the image measure of $\Pbb$ by 
\[(\gamma,u) \colon \Omega \to \SL_d(\Z) \ltimes \Tbb^d,\quad \omega \mapsto \bigl(x \in \Tbb^d \mapsto \gamma(\omega)x + u(\omega) \bigr)\]
is a finitely supported probability measure on $\SL_d(\Z) \ltimes \Tbb^d$. Conversely, every finitely supported probability measure on $\SL_d(\Z) \ltimes \Tbb^d$ can be realized in this way. 

For $u \colon \Omega \to \Tbb^d$, let \( H_u = \langle (\gamma,u)(\Omega) \rangle \) be the subgroup generated by $(\gamma,u)(\Omega) \subset \SL_d(\Z) \ltimes \Tbb^d$. 
\newchange{We will see in Lemma~\ref{lm:finiteorbit} that, under the assumption~\eqref{as:irreducible}, the orbit of a point $x \in \Tbb^d$ under the action of $H_u$ is finite if and only if
\[\forall \omega \in \Omega,\quad (\gamma,u)(\omega)x - x \in \Q^d/\Z^d.\]
}For $Q \geq 1$, we say a finite orbit $H_u x$ has \emph{height} at most $Q$ if there exists a positive integer $q \leq Q$ such that 
\[\forall \omega \in \Omega,\quad (\gamma,u)(\omega)x - x \in \frac{1}{q}\Z^d/\Z^d.\]
Let $\Pcal_Q$ denote the set of all $(u,x) \in (\Tbb^d)^\Omega \times \Tbb^d$ such that $H_u x$ is finite and has height at most $Q$.

Equip the space $(\Tbb^d)^\Omega \times \Tbb^d$ with the distance defined by 
\[d\bigl((u,x),(u',x')\bigr) = \max\bigl\{ \max_{\omega \in \Omega} d(u(\omega),u'(\omega)), d(x,x') \bigr\}\]
for any \((u,x), (u',x') \in (\Tbb^d)^\Omega \times \Tbb^d \).

For a probability measure $\mu_0$ on $\SL_d(\Z)$, we denote by $\lambda_{1,\mu_0}$ its top Lyapunov exponent, i.e.
\[\lambda_{1,\mu_0} = \lim_{n \to +\infty} \frac{1}{n}\int_{\SL_d(\Z)} \log \norm{g} \dd \mu_0^{*n}(g).\]
Recall that by a result of Furstenberg~\cite{Furstenberg}, \newchange{if the subgroup generated by the support of $\mu_0$ acts strongly irreducibly on $\R^d$ then $\lambda_{1,\mu_0} > 0$.}

\begin{thm}
\label{thm:affineTd}
Given $(\Omega, \Pbb)$ and $\gamma \colon \Omega \to \SL_d(\Z)$, let $\mu_0 = \gamma_*\Pbb$ and let $\Gamma$ denote the group generated by $\gamma(\Omega)$. 
Assume that $\Gamma$ satisfies \eqref{as:irreducible} and \newchange{either} \eqref{as:proximal} or \eqref{as:Zconnected}.
Given $\lambda \in {(0,\lambda_{1,\mu_0})}$, there exists $C = C(\gamma, \Pbb, \lambda) > 1$ such that the following holds.

Let $u \colon \Omega \to \Tbb^d$ and set $\mu = (\gamma,u)_*\Pbb$.
Let $x \in \Tbb^d$.
If for some $a \in \Z^d \setminus\{0\}$, $t \in {(0,\frac{1}{2})}$ and $n \geq C \log\frac{\norm{a}}{t}$, we have
\[\abs{\widehat{\mu^{*n}*\delta_x}(a)} \geq t,\]
then $d\bigl((u,x),\Pcal_Q\bigr) \leq e^{-\lambda n}$ where $Q = \bigl( \frac{\norm{a}}{t} \bigr)^C$. 
\end{thm}

\newchange{It is worth noting that the constants $\lambda$ and $C$ are independent of the translation part $u$, as compared to the previous result of Boyer~\cite[Theorem 1.4]{Boyer}.}

The above statement for some positive $\lambda$ would suffice to deduce Theorem~\ref{thm:qualitative}. We emphasize that this $\lambda$ can be made arbitrarily close to $\lambda_{1,\mu_0}$, but not larger, as shown by the following fact.

\begin{prop}
\label{pr:rate}
Given $(\Omega,\Pbb)$ and $\gamma \colon \Omega \to \SL_d(\Z)$, let $\mu_0 = \gamma_*\Pbb$. Given $\lambda > \lambda_{1,\mu_0}$, there exists $c = c(\mu_0,\lambda) > 0$ such that the following holds.
If $(u,x) \in (\Tbb^d)^\Omega \times \Tbb^d$ satisfies 
\[d\bigl((u,x),\Pcal_Q\bigr) \leq e^{-\lambda n}\]
for some $n \geq 1$ and $Q \geq 1$,
then there exists a positive integer $q \leq Q$ such that for any $a \in q \Z^d$,
\[\abs{\widehat{\mu^{*n}*\delta_x}(a)} \geq 1 - \norm{a} e^{-c n}\]
for $\mu = (\gamma,u)_*\Pbb$.
\end{prop}

\subsection{Outline of the proof and structure of the paper.}
Consider the random walk associated to the translations $u \in (\Tbb^d)^\Omega$ and the starting point $x$.
We divide the time into three parts $n = n_1 + n_2 + n_3$. Assume that the random walk has a large Fourier coefficient at time $n$ and moreover, for a contradiction, that the data $(u,x)$ is not close to that of a periodic orbit of small height.

First, we show that there is an initial non-concentration after time $n_1$. More precisely, $\mu^{n_1}*\delta_x$ does not concentrate in balls of radius $r> 0$ unless the data $(u,x)$ is $e^{Cn_1}r$ close to that of periodic orbit with height $\leq e^{Cn_1}$.
This is the objective of Section~\ref{sc:n1}.
The main idea is to express the property of the random walk being concentrated in terms of the tuple $(u,x)$ being close to a solution to a system of linear equations with integer coefficients.
By taking reduction modulo a prime number $p$ of the equations, we transform the problem into that of establishing non-concentration for affine random walks on the space $\Fbb_p^d$ over the prime field $\Fbb_p$. Such an estimate was a key component in work of Varj\'u and the third named author~\cite{LV} regarding spectral gap for the group of affine transformations on $\Fbb_p^d$.
A modification of these estimates more suitable to our needs is provided in Appendix~\ref{appendix}.

The time we spend in the middle regime is $n_2$. 
These additional iterations improve the initial non-concentration to a stronger almost optimal energy estimate that captures non-concentration at all scales. Here the non-concentration is to be understood as an upper-bound on the energy (defined in Section~\ref{sc:n2}) of the measure on $\Tbb^d$.
The idea is to consider the random walk on $\Tbb^d \times \Tbb^d$ and use a Margulis function to control the probability of getting too close to the diagonal.
This part is the goal of Section~\ref{sc:n2}. It turns out that to get a non-concentration at scale $\eta>r$ using this Margulis function technique one needs to take $n_2$ so that $\eta e^{-\lambda n_2} < r$, where $\lambda \in {(0,\lambda_{1,\mu_0})}$ (how close $\lambda$ is to $\lambda_1$ influences the implicit parameter of the energy we use).

Suppose now that at the last iterate, i.e. for $n=n_1+n_2+n_3$, the random walk has a large Fourier coefficient $\abs{\widehat{\mu^{*n}*\delta_x}(a)} \geq t$, with $a \in \Z^d\setminus\{0\}$ and $t > 0$. Then the proofs in~\cite{BFLM} and~\cite{HS} show respectively that, going back some $ C'\log\frac{\norm{a}}{t} $ steps in time, the measure $\mu^{*(n- C'\log\frac{\norm{a}}{t})}*\delta_x$ has a lot of large Fourier coefficients and consequently it has some granular structure, that can be further bootstrapped to yield that for $n_3=C\log\frac{\norm{a}}{t}$ (for suitable $C>C'$) there is a radius $\rho > 0$ of size roughly $e^{-c n_3}$ such that $\mu^{*(n-n_3)}*\delta_x$ has some concentration at scale $\rho$.
This part is explained in Section~\ref{sc:n3}.

Finally, with $re^{n_2\lambda} = \rho$, the above leads to a contradiction, finishing the proof of Theorem~\ref{thm:affineTd}. 
This together with the proof of Theorem~\ref{thm:qualitative} and of Proposition~\ref{pr:rate} are contained in Section~\ref{sc:mainresults}.

\begin{newchanges}
\section{Characterization of finite orbits.}

In the linear case, if $\Gamma \subset \SL_d(\Z)$ acts strongly irreducibly on $\R^d$, then it is a result of Guivarc'h-Starkov~\cite{GS} and of Muchnik~\cite{Muchnik} that every $\Gamma$-orbit in $\Tbb^d$ is either finite or dense and moreover it is finite if and only if it contains only rational points, or equivalently if and only if the starting point is rational. Note that a quantitative version of these results can be deduced from \cite{HS} (or \cite{BFLM} in the proximal case). The following lemma is the analogous characterization of finite orbits, justifying the definition of the height of a finite orbit we gave in the introduction.

\begin{lemm}
\label{lm:finiteorbit}
Let $H \subset \SL_d(\Z) \ltimes \R^d$ be a subgroup and $S$ a generating set.
Let $\Gamma$ be the image of $H$ in $\SL_d(\Z)$. Assume that $\Gamma$ acts strongly irreducibly on $\R^d$.
Then for any $x \in \Tbb^d$, the following are equivalent.
\begin{enumerate}
\item \label{it:finiteorbit} The orbit $Hx \subset \Tbb^d$ is finite.
\item \label{it:rationalgen} There exists $q \in \N$ such that for all $g \in S$, $g x - x \in \frac{1}{q}\Z^d/\Z^d$.
\item \label{it:rationalall} There exists $q \in \N$ such that for all  $g \in H$, $g x - x \in \frac{1}{q}\Z^d/\Z^d$.
\end{enumerate}
\end{lemm}

\begin{proof}
The implications \ref{it:rationalall} $\Rightarrow$ \ref{it:finiteorbit} and \ref{it:rationalall} $\Rightarrow$ \ref{it:rationalgen} are immediate.

For all $g = (\gamma,u) \in \SL_d(\Z)\ltimes \R^d$, all $h \in \SL_d(\Z)\ltimes \R^d$ and all $x \in \Tbb^d$, we have
\[ghx - x = \gamma(hx-x) + gx - x.\]
This allows to show by a simple induction  the implication \ref{it:rationalgen} $\Rightarrow$ \ref{it:rationalall}.

Finally, we show \ref{it:finiteorbit}  $\Rightarrow$ \ref{it:rationalall}. 
Assume that $Hx$ is a finite orbit.
After conjugating by the translation by $x$, we may assume without loss of generality that $x = 0 \in \Tbb^d$.
Then the stabilizer $\Stab_H(0)$ of $0$ in $H$ is $H \cap \SL_d(\Z)$ and has finite index in $H$. 
Hence $\Stab_H(0)$ has finite index in $\Gamma$, and therefore, by assumption, acts strongly irreducibly on $\R^d$. It also has a semisimple Zariski closure in $\SL_d(\R)$ by~\cite[Lemme 8.5]{BQ1}. Thus we can e.g.\ apply the results of Guivarc’h-Starkov and of Muchnik\footnote{Indeed, we only apply the easy part of these results characterizing finite orbits.} to the group $\Stab_H(0)$.
For every $g \in H$, the $\Stab_H(0)$-orbit of $g(0)$ is finite, hence by \cite[Theorem 1.2]{Muchnik}, $g(0)$ is rational.
This shows \ref{it:rationalall}.
\end{proof}
\end{newchanges}

\section{Initial non-concentration}
\label{sc:n1}
In this paragraph we prove a non-concentration estimate using knowledge about affine random walks on the Euclidean space $\R^d$ and that on the vector space $\Fbb_p^d$ over prime fields, which will be established in the next paragraphs. Since this does not involve additional difficulties, and may be useful for future extensions, we prove the results in this section in somewhat greater generality than we need.


\begin{prop}
\label{pr:initNC}
Given $(\Omega, \Pbb)$ and $\gamma \colon \Omega \to \SL_d(\Z)$. Assume that the group generated by $\gamma(\Omega)$ acts strongly irreducibly on $\Q^d$ and its Zariski closure is semisimple. 
Then there exists $C_1 = C_1(\Pbb,\gamma) > 1$ such that the following holds.
Let $u \colon \Omega \to \Tbb^d$ and set $\mu = (\gamma,u)_*\Pbb$.
For any integer $n_1 \geq C_1$, any radius $r > 0$ and any point $x \in \Tbb^d$, if 
\[\max_{y \in \Tbb^d}\  \mu^{*n_1}*\delta_x (B(y,r)) \geq e^{-\frac{n_1}{C_1}},\]
then 
\[d((u,x), \Pcal_Q ) \leq e^{C_1 n_1}r\]
where $Q = e^{C_1 n_1}$.
\end{prop}

To prove Proposition \ref{pr:initNC}, we shall use the following \newchange{elementary lemma and a proposition about affine random walks on $\R^d$, namely Proposition~\ref{pr:affineRd}}. The proof of this key proposition is deferred to \S\ref{sec:EuclidRW}.

\begin{lemm}
\label{lm:solZlin}
Let $D \geq 1$ and let $\Phi \subset (\R^D)^*$ be a collection of linear forms with integer coefficients. Let $M = \max_{\phi \in \Phi} \norm{\phi}$. There exists an integer $q \leq M^D$ such that   
\[\bigcap_{\phi \in \Phi} \phi^{-1}\bigl(\Z + B(0,r) \bigr) \subset \bigcap_{\phi \in \Phi} \ker \phi + \frac{1}{q}\Z^D + B(0,D^{\frac{D}{2}} M^{D-1} r).\]
\end{lemm}

\newchange{Given a map $(\gamma,u) \colon \Omega \to \SL_d(\Z) \ltimes \R^d$, we extend its definition to $\Omega^n$ for every $n \geq 1$ by setting for every $\omegatuple = (\omega_n,\dotsc, \omega_1) \in \Omega^n$,
\begin{equation}\label{eq:def_of_affine_composition}(\gamma,u)(\omegatuple) = (\gamma,u)(\omega_n) \dotsm (\gamma,u)(\omega_1) \in \SL_d(\Z) \ltimes \R^d.\end{equation}
Thus, the push-forward measure of $\Pbb^{\otimes n}$ by $(\gamma,u)$ is exactly the $n$-th convolution of $(\gamma,u)_*\Pbb$ with itself.}

\begin{prop}
\label{pr:affineRd}
Let $(\Omega, \Pbb)$ with $\Omega$ finite and $\gamma \colon \Omega \to \SL_d(\Z)$ such that the subgroup generated by $\gamma(\Omega)$ acts strongly irreducibly on $\Q^d$ and its Zariski closure is semisimple. \newchange{Then} there exists $C = C(\Pbb,\gamma) > 0$ such that the following holds for any
$u \colon \Omega \to \R^d$.
\newchange{For} any $x, y \in \R^d$,
either
\[\forall n \geq 1, \quad \Pbb^{\otimes n} ( \{ \omegatuple  \in \Omega^n \mid  (\gamma,u)(\omegatuple) x = y \} ) \leq C e^{-\frac{n}{C}}\]
or $x = y$, and moreover is a fixed point of the group generated by $(\gamma,u)(\Omega)$.
\end{prop}

\newchange{Now, we prove Proposition \ref{pr:initNC} using Lemma~\ref{lm:solZlin} and Proposition~\ref{pr:affineRd}.}
\begin{proof}[Proof of Proposition \ref{pr:initNC}.]
\begin{newchanges}
We lift everything to $\R^d$. It is enough to show that there exists a constant $C_1$ depending on $(\Omega,\Pbb)$ and the map $\gamma \colon \Omega \to \SL_d(\Z)$ such that if $n \geq C_1$, $r > 0$, $u \colon \Omega \to \R^d$, and $x,y\in \R^d$ satisfy
\begin{equation}
\label{eq:munxyrho}
\Pbb^{\otimes n}\{\, \omegatuple \in \Omega^n \mid (\gamma,u)(\omegatuple)x \in B(y,r) + \Z^d\,\} \geq e^{-\frac{n}{C_1}},
\end{equation}
then there exist an integer $q \leq e^{C_1 n}$, a map $u' \colon \Omega \to \R^d$ and a point $x' \in \R^d$ such that 
\begin{equation}
\label{eq:F1qZ}
\forall \omega \in \Omega, \quad \gamma(\omega) x' + u'(\omega) - x' \in \frac{1}{q}\Z^d
\end{equation}
and
\begin{equation}
\label{eq:uxux}
\max_{\omega \in \Omega} \norm{u'(\omega)-u(\omega)} \leq e^{Cn}r \quad \text{and} \quad \norm{x' - x} \leq e^{Cn}r.
\end{equation}

For each $\omegatuple \in \Omega^n$, consider the linear map $\phi_\omegatuple \colon (\R^d)^\Omega \times \R^d \times \R^d \to \R^d$ defined by
\[\phi_\omegatuple(u,x,y) = (\gamma,u)(\omegatuple)x - y.\]
Proposition~\ref{pr:affineRd} can be reformulated as: there is a constant $C > 0$ such that for any subset $W \subset \Omega^n$ with measure $\Pbb^{\otimes n}(W) > Ce^{-\frac{n}{C}}$, we have 
\[\bigcap_{\omegatuple \in W} \ker \phi_\omegatuple = F := \bigl\{\,(u,x,y) \in  (\R^d)^\Omega \times \R^d \times \R^d \mid (\gamma,u)(\Omega)x = x = y \,\bigr\}.\]

Now specialize to some $n\geq 1$, $r>0$ and $(u,x,y) \in (\R^d)^\Omega \times \R^d \times \R^d$ satisfying \eqref{eq:munxyrho}. Choose
\[W = \{\, \omegatuple \in \newchange{\Omega^n} \mid \phi_\omegatuple(u,x,y) \in \Z^d + B(0,r) \,\}\]
so that \eqref{eq:munxyrho} translates to $\Pbb^{\otimes n} (W) \geq e^{-\frac{n}{C_1}}$.
Thus, if $C_1$ is chosen large enough, by the reformulation of Proposition~\ref{pr:affineRd},
\[\bigcap_{\omegatuple \in W} \ker \phi_\omegatuple = F \]

By the definition of $W$, 
\[(u,x,y) \in \bigcap_{\omegatuple \in W} \phi_\omegatuple^{-1}\bigl(\Z^d + B(0,r)\bigr). \]
Note that the linear maps $\phi_\omegatuple$ have integer coefficients in the standard basis and the coefficients can be bounded: $\forall \omegatuple \in \Omega^n$, $\norm{\phi_\omegatuple} \leq e^{Cn}$ for some $C> 0$ depending only on $\gamma$. It follows from Lemma~\ref{lm:solZlin} that there exists a positive integer $q \leq e^{C_1n}$ such that 
\[(u,x,y) \in \bigcap_{\omegatuple \in W} \ker \phi_\omegatuple + (\frac{1}{q} \Z^d)^\Omega \times \frac{1}{q} \Z^d \times \frac{1}{q} \Z^d. + B(0,e^{C_1n}r)\]
for some large constant $C_1$.
Hence 
\[(u,x,y) \in F + (\frac{1}{q} \Z^d)^\Omega \times \frac{1}{q} \Z^d \times \frac{1}{q} \Z^d. + B(0,e^{C_1n}r),\]
which implies existence of $(u',x') \in (\R^d)^\Omega \times \R^d$ satisfying \eqref{eq:F1qZ} and \eqref{eq:uxux}.
\end{newchanges}
\end{proof}

\begin{proof}[Proof of Lemma \ref{lm:solZlin}.]
Let $K = \bigcap_{\phi \in \Phi} \ker \phi$. 
First, we prove the lemma for the special case where $K = \{0\}$. 
In this case we can find $\phi_1,\dotsc,\phi_D \in \Phi$ such that $A = (\phi_1,\dotsc, \phi_D) \colon \R^D \to \R^D$ is an invertible endomorphism and its matrix in the standard basis has coefficients in $\Z$. 
Hence $q = \abs{\det(A)}$ is an integer satisfying $1 \leq q \leq M^D$.
The inclusion
\[ \bigcap_{i=1}^D \phi_i^{-1}\bigl(\Z + B(0,r) \bigr) \subset  \frac{1}{q}\Z^D + B(0,D^{\frac{D}{2}} M^{D-1} r)\]
follows  immediately.




Now consider the general case. Let $k = \dim K$. After permuting the coordinates if necessary, we may assume that $\R^D = \R^{D - k} \oplus K$ where $\R^{D - k}$ denotes the subspace corresponding to the first $D - k$ coordinates. Applying the special case to the collection
\[\{\, \phi_{\mid \R^{D-k}}  \mid \phi \in \Phi\,\} \subset  (\R^{D-k})^* \]
yields the lemma.
\end{proof}

\subsection{Affine random walks on the Euclidean space}\label{sec:EuclidRW}

Now we turn to prove Proposition \ref{pr:affineRd}. 
The idea is to reduce to the following analogous statement about affine random walk on $\Fbb_p^d$.
\begin{prop}
\label{pr:affineFpd}
Let $(\Omega, \Pbb)$ with $\Omega$ finite and $\gamma \colon \Omega \to \SL_d(\Z)$ such that the subgroup generated by $\gamma(\Omega)$ acts strongly irreducibly on $\Q^d$ and its Zariski closure is semisimple. \newchange{Then} there exists $C = C(\Pbb,\gamma) > 0$ such that the following holds for any prime number $p$ and any
$u \colon \Omega \to \Fbb_p^d$.
\newchange{For} any $x, y \in \Fbb_p^d$, 
either
\[\forall n \geq 1 , \quad \Pbb^{\otimes n} \bigl( \bigl\{\, \newchange{\omegatuple \in \Omega^n} \mid  (\gamma,u)(\omegatuple) x = y \,\bigr\} \bigr)  \leq C \max\{p^{-1/4},e^{-\frac{n}{C}}\}\]
or $x = y$, and moreover is the unique fixed point of the group generated by $(\gamma,u)(\Omega)$.
\end{prop}
\newchange{Here, the map $(\gamma,u) \colon \Omega \to \SL_d(\Z) \ltimes \Fbb_p^d$ is extended to $\Omega^n$ in the same manner as in equation \eqref{eq:def_of_affine_composition}.} 
This result is largely based on the work of Varj\'u and the third named author~\cite{LV}.
We postpone the proof to the appendix (see Section~\ref{proof of prop. 2.3}).
Here, we deduce Proposition~\ref{pr:affineRd} from Proposition~\ref{pr:affineFpd}.

\begin{proof}[Proof of Proposition~\ref{pr:affineRd} (assuming Proposition~\ref{pr:affineFpd}).]
We use the notation in the proof of Proposition~\ref{pr:initNC}. We need to establish the equality between the linear subspaces
\[K := \bigcap_{\omegatuple \in W} \ker \phi_\omegatuple \]
and
\[F := \bigl\{\,(u,x,y) \in  (\R^d)^\Omega \times \R^d \times \R^d \mid (\gamma,u)(\Omega)x = x = y \,\bigr\},\]
whenever $W \subset \Omega^n$ satisfies $\Pbb^{\otimes n}(W) > Ce^{-\frac{n}{C}}$, where $C = C(\Pbb,\gamma)$ is the constant given by Proposition~\ref{pr:affineFpd}.

The subspace $K$, being the kernel of a matrix with integer coefficients is rational and hence spanned by integral vectors i.e. by $K \cap \bigl((\Z^d)^{\Omega} \times \Z^d \times \Z^d\bigr)$.

Let $(u,x,y) \in K \cap \bigl((\Z^d)^{\Omega} \times \Z^d \times \Z^d\bigr) $ . We show that $(u,x,y) \in F$.
Let $p$ be an arbitrary prime number. 
Denote by $\pi_p \colon \Z^d \to \Fbb_p^d$ the reduction modulo $p$.
Taking the reduction modulo $p$ of the relation $\phi_\omegatuple(u,x,y) = 0$, $\forall \newchange{\omegatuple} \in W$, we find that 
\[ \Pbb^{\otimes n} ( \{ \newchange{\omegatuple} \in \Omega^n \mid (\gamma,\pi_p(u))(\newchange{\omegatuple}) \pi_p(x) = \pi_p(y) \} ) > Ce^{-\frac{n}{C}}.\]
Thus, if $p$ is large enough, by Proposition~\ref{pr:affineFpd}, we have
\[\forall \omega \in \Omega,\quad (\gamma,u)(\omega)x \equiv x \equiv y \mod p.\]
This being true for all sufficiently large primes $p$, we deduce that 
\[\forall \omega \in \Omega,\quad (\gamma,u)(\omega)x = x = y.\]
Consequently, $(u,x,y) \in F$.
Hence $K \subset F$ and this finishes the proof.
\end{proof}

\section{Improving the initial non-concentration using a Margulis function}
\label{sc:n2}
Let $\alpha > 0$ be a parameter.
For a Borel measure $\nu$ on $\Tbb^d$, we define its $\alpha$-\emph{energy} to be
\[\Ecal_\alpha(\nu) = \iint_{\Tbb^d \times \Tbb^d \setminus \Delta} \frac{\dd \nu(x) \dd \nu(y)}{d(x,y)^\alpha}\]
where $\Delta$ denotes the diagonal
\[\Delta = \{\, (x,y) \in \Tbb^d \times \Tbb^d \mid x = y \,\}.\]
The objective of this section is the following.
\begin{prop}
\label{pr:MargulisF}
Let $\mu$ be a finitely supported probability measure on $\SL_d(\Z) \ltimes \Tbb^d$.
Let $\mu_0$ be its push-forward to $\SL_d(\Z)$.
Let $\Gamma \subset \SL_d(\Z)$ denote the subgroup generated by the support of $\mu_0$ and let $\lambda_{1,\mu_0}$ denote the top Lyapunov exponent of $\mu_0$.

Assume that $\Gamma$ acts \newchange{strongly} irreducibly on $\R^d$.
Then, given $\lambda \in {(0,\lambda_{1,\mu_0})}$, there exist constants $\alpha = \alpha(\mu_0,\lambda) > 0$ and  $C_2 = C_2(\mu_0,\lambda) > 1$ such that the following holds. For any Borel measure $\nu$ on $\Tbb^d$, any integer $n_2 \geq C_2$ and any radius $\rho > 0$,
\[\max_{y \in \Tbb^d}\, \mu^{* n_2}*\nu \bigl( B(y,\rho) \bigr)^2 \leq \nu \otimes \nu (\Delta) + 2^\alpha \rho^\alpha \bigl( e^{-\alpha \lambda n_2} \Ecal_\alpha(\nu) + C_2 \bigr).\]
\end{prop}

Notice that, under our assumption, $\Gamma$ is not \newchange{relatively} compact. Hence by a result of Furstenberg~\cite{Furstenberg}, $\lambda_{1,\mu_0}> 0$.

When a group $G$ acts on a topological space $X$ and $\mu$ is a measure on $G$, a function \( u \colon X \to \R_+ \) is said to satisfy the \emph{contraction hypothesis} for the associated random walk if it is proper and there \newchange{exist} an integer $m \geq 1$ and constants $0< a < 1$ and $C > 0$ such that
\begin{equation}
\label{eq:MargulisF}
\forall x \in X,\quad \int_{G} u(gx) \dd \mu^{*m}(g) \leq a u(x) + C.
\end{equation}
To keep track of the constants, we say that $u$ satisfies $\CH(m,a,C)$. \newchange{In the homogeneous dynamics context the use of such inequalities has been introduced by Margulis (see e.g. \cite{EskinMargulis, EskinMargulisMozes}); functions satisfying such inequalities are also known as Lyapunov functions.}

Here we consider $G = \SL_d(\Z) \ltimes \Tbb^d$ acting diagonally on $X = \Tbb^d \times \Tbb^d \setminus \Delta$. 
\begin{lemm}
\label{lm:MargulisF}
Under the assumption of Proposition~\ref{pr:MargulisF}, given $\lambda \in {(0,\lambda_{1,\mu_0})}$ there \newchange{exist} $\alpha > 0$, $m \geq 1$ and $C > 0$ depending only on $\mu_0$ and $\lambda$ such that the function on $\Tbb^d \times \Tbb^d \setminus \Delta$
\[(x,y) \mapsto d(x,y)^{-\alpha}\]
satisfies $\CH(m,e^{-\alpha\lambda m},C)$ for the random walk associated to $\mu$. Consequently, there exists a constant $C = C(\mu_0,\lambda) > 0$ \newchange{such that} for all $n \geq C$, 
\[\forall x \neq y \in \Tbb^d,\quad \int_G d(gx,gy)^{-\alpha} \dd \mu^{*n}(g) \leq e^{-\alpha\lambda n} d(x,y)^{-\alpha} + C.\]
\end{lemm}

This is essentially contained in~\cite{EskinMargulis}. We reproduce the proof here to highlight that the ratio $a$ in the property~\eqref{eq:MargulisF} can be made arbitrary close to $e^{-\alpha m \lambda_{1,\mu_0}}$.

\begin{proof}
Note that for every $g \in G$ with linear part $g_0 \in \SL_d(\Z)$, we have
\[\forall x, y \in \Tbb^d, \quad d(gx,gy) = d(g_0x, g_0y) = d(g_0(x-y),0).\]
This consideration allows us to reduce to the case of $\SL_d(\Z)$ acting on $\Tbb^d \setminus \{0\}$ with the random walk being defined by $\mu_0$.

We aim to establish existence of $\alpha > 0$ and $C > 1$ such that for all integers $n$ large enough,
\begin{equation}
\label{eq:bestMF}
\forall x \in \Tbb^d \setminus \{0\} ,\quad \int_{\SL_d(\Z)} d(gx,0)^{-\alpha} \dd \mu_0^{* n}(g)  \leq e^{-\alpha \lambda n} d(x,0)^{-\alpha} + C.
\end{equation}

By the law of large number for the norm cocycle \cite{Furstenberg} (see also \cite[Theorem 4.28(d)]{BQ_book}), there exists $m \geq 1$ such that
\[\forall v \in \R^d\setminus \{0\},\quad  \int_{\SL_d(\Z)} \log \frac{\norm{gv}}{\norm{v}} \dd \mu_0^{*m}(g) > m \newchange{\frac{\lambda + \lambda_{1,\mu_0}}{2}}.\]
\newchange{Fixing this $m$ and} using the inequality $\forall t\in \R$, $e^t \leq 1 + t + \frac{t^2}{2}e^{\abs{t}}$, we obtain, for $\alpha > 0$ small enough,
\[\sup_{v \in \R^d\setminus \{0\}}  \int_{\SL_d(\Z)} \Bigl(\frac{\norm{gv}}{\norm{v}}\Bigr)^{-\alpha} \dd \mu_0^{*m}(g) < e^{-\alpha \lambda m}.\]

Since $\mu_0$ is finitely supported, the quantity
\(M := \sup_{g \in \Supp(\mu_0)} \max \bigl\{ \norm{g}, \norm{g^{-1}} \bigr\} \) is finite.
On the one hand, if $x \in \Tbb^d$ is the projection of $v \in \R^d$ with $v \in B(0,\frac{1}{2M^m})$, then $d(gx,0) = \norm{gv}$ for all $g \in \Supp(\mu_0^{*m})$. On the other hand, if $x \in \Tbb^d \setminus B(0,\frac{1}{2M^m})$, then $d(gx,0) \geq \frac{1}{2M^{2m}}$ for all $g \in \Supp(\mu_0^{*m})$.
We obtain \eqref{eq:bestMF} for $n = m$ with $C=(2M^{2m})^\alpha$.

Then, by a simple induction, we establish \eqref{eq:bestMF} for all multiples of $m$ with slightly larger $C$.

Finally, for general $n$, write $n = km + l$ with $0 \leq l < m$. We have
\[\forall x \in \Tbb^d \setminus \{0\},\, \forall h \in \Supp(\mu_0^{*l}), \quad 
d(hx,0)^{-\alpha} \leq M^{\alpha m} d(x,0)^{-\alpha}.\]
Combined with \eqref{eq:bestMF} for $n - l = mk$, this implies that for every $x \in \Tbb^d \setminus \{0\}$,
\begin{align*}
\int d(gx,0)^{-\alpha} \dd \mu_0^{* n}(g) &= 
\iint d(ghx,0)^{-\alpha} \dd \mu_0^{* (n - l)}(g) \dd \mu_0^{* l}(h) \\
&\leq e^{-\alpha \lambda (n - l)} M^{\alpha m} d(x,0)^{-\alpha} + C.
\end{align*}
By repeating the argument with $\lambda' = \frac{\lambda + \lambda_{1,\mu_0}}{2}$ instead of $\lambda$, we may replace $\lambda$ by $\lambda'$ in the last inequality. Thus, \eqref{eq:bestMF} holds for all $n$ large enough.
\end{proof}

\begin{proof}[Proof of Proposition~\ref{pr:MargulisF}]
Let $B \subset \Tbb^d$ be a ball of radius $\rho > 0$.
We have
\[\mu^{*n}*\nu(B) = \int_G \int_{\Tbb^d} \indic_B(gx) \dd\nu(x) \dd\mu^{*n}(g)\]
Thus, by the Cauchy-Schwarz inequality,
\begin{align*}
\mu^{*n}*\nu(B)^2 &\leq \int_G \Bigl( \int_{\Tbb^d} \indic_B(gx) \dd\nu(x)\Bigr)^2 \dd\mu^{*n}(g)\\
&\leq \int_G \int_{\Tbb^d \times \Tbb^d} \indic_B(gx) \indic_B(gy) \dd\nu(x) \dd\nu(y) \dd\mu^{*n}(g)
\end{align*}
If $gx$ and $gy$ both belong to $B$ then either $x = y$ or $0 < d(gx,gy) < 2\rho$. Hence
\[\indic_B(gx) \indic_B(gy) \leq \indic_{\Delta}(x,y) + \frac{2^\alpha \rho^\alpha}{d(gx,gy)^\alpha}\indic_{\Tbb^d \times \Tbb^d \setminus \Delta}(x,y).\]
Therefore,
\[\mu^{*n}*\nu(B)^2 \leq \nu\otimes\nu(\Delta) +  2^\alpha \rho^\alpha \int_{\Tbb^d \times \Tbb^d \setminus \Delta} \int_G \frac{\dd\mu^{*n}(g) \dd \nu(x) \dd \nu (y)}{d(gx,gy)^{\alpha}} .\]
We conclude by using Lemma~\ref{lm:MargulisF}.
\end{proof}

\section{Large Fourier coefficients and granular structure}
\label{sc:n3}
The aim of this section is the following.
\begin{prop}
\label{pr:Fourier}
Let $\eta$ be a Borel probability measure on $\Tbb^d$.
Let $\mu$ be a Borel probability measure on $\SL_d(\Z) \ltimes \Tbb^d$.
Let $\mu_0$ be its push-forward to $\SL_d(\Z)$.
Let $\Gamma \subset \SL_d(\Z)$ denote the subgroup generated by the support of $\mu_0$.
Assume that $\mu_0$ has a finite exponential moment and $\Gamma$ satisfies assumption \eqref{as:irreducible} and either assumption \eqref{as:proximal} or assumption \eqref{as:Zconnected}.
Then given $\beta > 0$, there exists $C_3 = C_3(\mu_0,\beta) > 0$ such that the following holds.
If for some $a \in \Z^d \setminus\{0\}$, some $t \in {(0,\frac{1}{2})}$ and some $n_3 \geq C_3 \log\frac{\norm{a}}{t}$, we have
\[ \bigl\lvert  \widehat{\mu^{*n_3} * \eta}(a) \bigr\rvert \geq t,\]
then there exists $y \in \Tbb^d$ such that
\[\eta(B(y,\rho)) \geq \rho^\beta\]
for some $e^{- C_3 n_3} \leq \rho \leq e^{-\frac{n_3}{C_3}}$.
\end{prop}
Here is where we use the arguments developed in the linear case. The proof consists of two parts.
\subsection{Initial granulation estimate}
First, we show that if after some steps the random walk has a large Fourier coefficient, then there were a lot of large Fourier coefficients earlier in time, which in turn implies that the distribution had a granular structure. 
This part corresponds to the Phase I of the proof in~\cite{BFLM}. All we do here is to remark that the arguments for the linear random walk work also for the affine case.
\begin{prop}
\label{pr:granul}
We use the notation $\eta$, $\mu$, $\mu_0$ as in the statement of Proposition~\ref{pr:Fourier}.
Under the same assumption on $\mu_0$, there exist constants $C > 1$ and $\sigma > \tau > 0$ depending only on $\mu_0$ such that the following hold.
If for some $a \in \Z^d \setminus\{0\}$, some $t \in {(0,\frac{1}{2})}$ and some $m \geq C \abs{\log t}$, we have
\[ \bigabs{\widehat{\mu^{*m} * \eta}(a)} \geq t,\]
then there exists a $r_0$-separated subset $X \subset \Tbb^d$ such 
\[\eta \Bigl( \bigcup_{x \in X}B(x,\rho_0)\Bigr) \geq t^C\]
where
\[\rho_0 = e^{-\sigma m} \norm{a}^{-1} \quad \text{and} \quad r_0 = e^{\tau m} \rho_0.\]
\end{prop}

\begin{newchanges}
First remark the following relationship between the Fourier coefficients.
For any $g = (\gamma,u) \in \SL_d(\Z) \ltimes \Tbb^d$ and any $a \in \Z^d$, we have
\begin{equation*}
\widehat{g_*\eta}(a) = e^{2\pi i \langle a,u \rangle}\widehat{\gamma_*\eta}(a).
\end{equation*} 
This can be used to prove the following affine version of \cite[Lemma 4.3]{HS}, 
\begin{lemm}
\label{lm:mu0kA}
Let $\mu$ be a Borel probability measure on $\SL_d(\Z) \ltimes \R^d$ and $\eta$ a Borel probability measure on $\Tbb^d$. Let $\mu_0$ be the push-forward of $\mu$ to $\SL_d(\Z)$.
Assume for some $a \in \Z^d\setminus\{0\}$, and some $t > 0$,
\[\abs{\hat{\mu *\eta}(a)} \geq t.\]
Then for any integer $k \geq 1$, the set
\[A = \Bigl\{\,\gamma \in \Mat_d(\Z) \mid \bigabs{\hat{\eta}(\gamma^{\tr} a)} \geq \frac{t^{2k}}{2}\,\Bigr\}\]
satisfies 
\[\mu_0^{(k)}(A) \geq \frac{t^{2k}}{2}\]
where $\Mat_d(\Z)$ denotes the set of $d \times d$ matrices with coefficients in $\Z$ and $\mu_0^{(k)}$ is the push-forward measure of $\mu_0^{\otimes 2k}$ by the map \[(\gamma_1,\dotsc,\gamma_{2k}) \in \Mat_d(\Z)^{2k} \mapsto \gamma_1 + \dotsb + \gamma_k - \gamma_{k+1} - \dotsb - \gamma_{2k} \in \Mat_d(\Z).\]
\end{lemm}
\begin{proof}
We have
\begin{align*}
\widehat{\mu * \eta}(a) &= \int_{\SL_d(\Z)\ltimes \R^d} \widehat{g_*\eta}(a) \dd \mu(g)\\
&= \int_{\SL_d(\Z)\ltimes \R^d} \int_{\Tbb^d} e^{2\pi i \langle a, \gamma x + u\rangle} \dd \eta(x) \dd \mu(\gamma,u).
\end{align*}
Using Fubini's theorem to exchange the order of integration and then H\"older's inequality, we obtain 
\begin{align*}
&\bigabs{\widehat{\mu * \eta}(a)}^{2k}\\
\leq& \int_{\Tbb^d} \bigabs{ \int_{\SL_d(\Z)\ltimes \R^d} e^{2\pi i \langle a, \gamma x + u\rangle} \dd \mu(\gamma,u)}^{2k} \dd \eta(x)\\
=& \iint e^{2\pi i \langle a, \gamma_1 x + \dotsb + \gamma_k x - \gamma_{k+1} x - \dotsb - \gamma_{2k} x + u_1 + \dotsb + u_k - u_{k+1} - \dotsb - u_{2k}\rangle} \\
&\hspace{17.3em} \dd \mu^{\otimes 2k}((\gamma_1,u_1), \dotsc, (\gamma_{2k},u_{2k})) \dd \eta(x)\\
=& \int e^{2\pi i \langle a, u_1 + \dotsb + u_k - u_{k+1} - \dotsb - u_{2k}\rangle}\hat\eta \bigl((\gamma_1 + \dotsb + \gamma_k - \gamma_{k+1} - \dotsb - \gamma_{2k})^{\tr}a\bigr) \\
&\hspace{20em} \dd \mu^{\otimes 2k}((\gamma_1,u_1), \dotsc, (\gamma_{2k},u_{2k}))\\
\leq& \int \bigabs{\hat\eta \bigl((\gamma_1 + \dotsb + \gamma_k - \gamma_{k+1} - \dotsb - \gamma_{2k})^{\tr}a\bigr)}  \dd \mu_0^{\otimes 2k}(\gamma_1, \dotsc, \gamma_{2k})
\end{align*}
This yields the desired estimate.
\end{proof}
\end{newchanges}

\begin{proof}[Proof of Proposition~\ref{pr:granul}]
\begin{newchanges}
In the case where we assume~\eqref{as:Zconnected}, 
by replacing \cite[Lemma 4.3]{HS} by the previous lemma, the same argument in the proof of \cite[Proposition 4.1]{HS} works.

For the reader's convenience, we briefly summarize the argument. Assume $\bigabs{\widehat{\mu^{*m} * \eta}(a)} \geq t$ for some $t \in {(0, \frac{1}{2})}$ and some $m \geq C \abs{\log t}$.
In view of \cite[Proposition 7.5]{BFLM}, it is enough to show that there is a constant $C$ depending only on $\mu_0$ such that 
\[\Ncal\bigl(\{\, a \in \Z^d \cap B(0,\rho_0^{-1}) \mid \abs{\hat{\eta}(a)} \geq t^C \,\}, r_0^{-1}\bigr) \geq t^C \Bigl(\frac{r_0}{\rho_0}\Bigr)^d\]
where for $A \subset \R^d$ and $\delta > 0$,
$\Ncal(A,\delta)$ denotes the $\delta$-covering number of $A$ i.e. the least number of balls of radius $\delta$ that cover $A$.

By Lemma~\ref{lm:mu0kA} applied to the convolution $\mu^{*m}*\eta$, for any integer $k \geq 1$, there is a subset of matrices $A \subset \Mat_d(\Z)$ such that $(\mu_0^{*m})^{(k)}(A) \gg t^{2k}$
and for all $\gamma \in A$, $\abs{\hat\eta(\gamma^{\tr}a)} \gg t^{2k}$.
From \cite[Theorem 3.19]{HS}, we deduce that for a suitably chosen large $k$, at some suitable scale, the convolution $(\mu_0^{*m})^{(k)}$ is a "flattened" measure on the subring of $\Mat_d(\Z)$ generated by $\Supp(\mu_0)$. 
This will imply that the set $\{\, \gamma^{\tr} a \in \Z^d \mid \gamma \in A\,\}$ has a very large covering number at some suitable scale.

In the case where we assume~\eqref{as:proximal}, the adaptation is similar. Instead of \cite[Proposition 4.1]{HS}, we follow the proof of \cite[Proposition 7.1]{BFLM}.
\end{newchanges}
\end{proof}

\subsection{Bootstrapping the concentration}
Once we have know there is a granular structure, we use a bootstrapping procedure to intensify the concentration by going back further in time.
\begin{prop}
\label{pr:bootstrap}
We use the notation $\eta$, $\mu$, $\mu_0$, $\Gamma$ as in the statement of Proposition~\ref{pr:Fourier}. Assume that $\mu_0$ has a finite exponential moment and that $\Gamma$ satisfies \eqref{as:irreducible}. Then given $\epsilon > 0$ there exists $C > 1$ (uniform on the translation part of the random walk measure $\mu$) such that the following holds. For any integer $l > C$ and real numbers $r,\rho > 0$ such that $r > e^{(d + 1)\lambda_1 l}\rho$, for every $r$-separated subset $X \subset \Tbb^d$, there exists a $r'$-separated subset $X' \subset \Tbb^d$ of cardinality at most that of $X$ such that
\[\eta\bigl(\bigcup_{x \in X'}B(x,\rho') \bigr) \geq \mu^{*l}*\eta\bigl(\bigcup_{x \in X}B(x,\rho) \bigr)^d - e^{-\frac{l}{C}}\]
where $r' = e^{-(\lambda_1 +\epsilon)l}r$ and $\rho' = e^{-(\lambda_1 -\epsilon)l} \rho$.
\end{prop}
\begin{newchanges}
Specialized to the linear case, this is noting else but \cite[Lemma 5.3]{HS}, which is a generalization of \cite[Proposition 7.2]{BFLM}.
\end{newchanges}
\begin{proof}
\begin{newchanges}
It is straightforward to adapt the argument in  \cite[Lemma 5.3]{HS} to our affine random walks for the following reason.

By a simple use of Jensen's inequality, we can show that
\[\mu^{*l}*\eta(X^{(\rho)})^d \leq \sum_{g_1,\dotsc,g_d \in \SL_d(\Z^d)\ltimes \R^d } \mu^{*l}(g_1) \dotsm \mu^{*l}(g_d)\eta\bigl( g_1^{-1}(X^{(\rho)}) \cap \dotsb g_d^{-1}(X^{(\rho)}) \bigr) .\]
where $X^{(\rho)} = \bigcup_{x \in X}B(x,\rho)$.

Thus, the key point is to show that given sufficiently small $\epsilon > 0$ there is $c = c(\mu_0,\epsilon)$ such that the following holds for $l \geq 1$ sufficiently large: if $g_1, \dotsc, g_d$ are independent random affine transformations of $\R^d$ distributed according to $\mu^{*l}$, then with probability greater than $
1 - e^{-cl}$, the intersection 
\(g_1^{-1}(X^{(\rho)}) \cap \dotsb g_d^{-1}(X^{(\rho)}) \) is contained in a union of at most $\abs{X}$ $r'$-separated balls of radius $\rho'$.

The statement about the intersection will follow from the following three properties (which hold with probability at least $1-e^{-cl}$).
\begin{itemize}
\item For any given $x_1 \in X$, for any $i = 2,\dotsc,d$, $g_ig_1^{-1}(B(x_1,\rho))$ is a contained in a ball of radius $e^{(d+ 1)\lambda_1 l}\rho < r$.
Hence, since $X$ is $r$-separated, the map $(x_1,\dotsc,x_d) \in X^d \mapsto x_1 \in X$ is injective when restricted to the subset of $d$-tuples $(x_1,\dotsc,x_d)$ such that $g_1^{-1}(B(x_1,r)) \cap \dotsb \cap g_d^{-1}(B(x_d,r))$ is non-empty.
\item For all $y,z \in \R^d$, if for all $i = 1,\dotsc,d$, $\norm{g_i(y) - g_i(z)}\leq 2\rho$ then $\norm{y - z} \leq e^{-(\lambda_1 - \epsilon) l} \rho = \rho'$.
Hence, for a given $d$-tuple
 $(x_1,\dotsc,x_d) \in X^d$, the intersection  $g_1^{-1}(B(x_1,r)) \cap \dotsb \cap g_d^{-1}(B(x_d,r))$ is, if not empty, contained in a ball of radius $\rho'$.
\item For all $y,z \in \R^d$, if $\norm{g_1(y) - g_1(z)} \geq r - 2\rho$ then $\norm{y -z} \geq r'+2\rho'$. Hence, in light of the first point, the $\rho'$-balls obtained in the second point are $r'$-separated.
\end{itemize}
All the three points holds for $(g_1,\dotsc,g_d)$ if and only if the corresponding linear parts of $(g_1,\dotsc g_d)$ satisfy the same properties.
The linear case is proven in \cite[Lemma 5.3]{HS}. Informally, the main argument there is that with high probability, the linear maps $g_1, \dotsc, g_d$ and their inverses all have the "expected" norm and the "big axes" of the ellipsoids $g_i^{-1}(B(x_i,r))$, $i = 1, \dotsc, d$ are in "transversal" position. 
\end{newchanges}
\end{proof}

Proposition~\ref{pr:Fourier} is then obtained by applying Proposition~\ref{pr:granul} and then iterate Proposition~\ref{pr:bootstrap}. The argument is identical to the relevant part in \cite[Proof of Proposition 3.1]{BFLM}.

\section{Proof of the main results}
\label{sc:mainresults}
Now we are ready to prove the main results.
\subsection{Quantitative statement}
\begin{proof}[Proof of Theorem~\ref{thm:affineTd}]
As in the statement, assume
\[\bigl\lvert  \widehat{\mu^{*n} * \delta_x}(a) \bigr\rvert \geq t\]
for some $a \in \Z^d \setminus \{0\}$, some  $t \in {(0,\frac{1}{2})}$ and some $n \geq C\log\frac{\norm{a}}{t}$. Here $C$ is a large constant to be determined.

Recall that $\mu = (\gamma,u)_*\Pbb$. Its linear part is $\mu_0 = \gamma_*\Pbb$.
Let $C_1 = C_1(\Pbb,\gamma) > 1$ be the constant given by Proposition~\ref{pr:initNC}.
We choose $\lambda' = \frac{\lambda + \lambda_1}{2}$ and
let $\alpha = \alpha(\mu_0,\lambda') > 0$ and $C_2 = C_2(\mu_0,\lambda') > 1$ be the constants given by Proposition~\ref{pr:MargulisF}.
We choose $\beta = \frac{\alpha}{3}$ and let $C_3 = C_3(\mu_0,\beta) > 1$ be the constant given by Proposition~\ref{pr:Fourier}.

We divide the $n$ random walk steps into three time periods, starting with $n_1 = L_1 \log\frac{\norm{a}}{t}$ random steps in the first period, then $n_2$ steps, and ending with $n_3 = L_3 \log\frac{\norm{a}}{t}$ steps, where $n = n_1 + n_2 + n_3$. The values of $L_1$ and $L_3$ are to be determined at the end of the proof, and they will only depend on $\gamma$, $\Pbb$ and $\lambda$.

We use the shorthand $\nu = \mu^{*n_1}*\delta_x$.
Set $r = e^{-\lambda' n_2}$. According to Proposition~\ref{pr:initNC}, either
\begin{equation}
\label{eq:initNC}
\max_{z \in \Tbb^d}\  \nu(B(z,r)) \leq e^{-\frac{n_1}{C_1}}
\end{equation}
or 
\[d((u,x),\Pcal_Q) \leq e^{C_1 n_1} r \quad \text{with} \quad Q = e^{C_1n_1} = \bigl( \frac{\norm{a}}{t} \bigr)^{C_1 L_1}.\]
In the latter case, we are done because we can guarantee
\(e^{C_1 n_1} r \leq e^{-\lambda n}\)
by requiring $C \geq (\lambda' - \lambda)^{-1}\bigl((\lambda' + C_1)L_1 + \lambda' L_3 \bigr)$.

Now assume~\eqref{eq:initNC} and we will deduce a contradiction. 
Applying Proposition~\ref{pr:Fourier} to the parameter $\beta$ and the measure $\eta = \mu^{*n_2} * \nu = \mu^{*(n_1 + n_2)}*\delta_x$, we obtain $y \in \Tbb^d$ such that
\begin{equation}
\label{eq:lateC}
\mu^{*n_2} * \nu \bigl(B(y,\rho)\bigr) \geq \rho^\beta    
\end{equation}
for some
\begin{equation}
\label{eq:rangerho}
\rho \in {[e^{-C_3n_3}, e^{-\frac{n_3}{C_3}}]}.
\end{equation}

In order to show that \eqref{eq:initNC} and \eqref{eq:lateC} are not compatible, we decompose
 $\nu$ into measures whose supports are $r$-separated, using the following lemma:
 
 \begin{lemm}
\label{lm:decomp}
Let $r > 0$. Let $\nu$ be a Borel probability measure on $\Tbb^d$ satisfying
\[\max_{z \in \Tbb^d}\ \nu(B(z,r)) \leq s\]
for some $s > 0$.
Given a measurable function $f \colon \Tbb^d \to {[0,1]}$, there exists a probability measure $\nu'$ on $\Tbb^d$ whose support is $r$-separated (hence finite) and such that
\[\int_{\Tbb^d} f \dd \nu' \geq 2^{-d} \int_{\Tbb^d} f \dd \nu \quad \text{and} \quad \nu' \otimes \nu'(\Delta) \leq \frac{2^d s}{\int_{\Tbb^d} f \dd \nu}.\]
\end{lemm}

\begin{proof}

Using a variant of the checkerboard tiling, we can construct easily a partition 
\[\Tbb^d = \bigsqcup_{i \in I,j\in J} Q_{i,j}\] 
indexed by $I \times J = \{1,\dotsc, 2^d\} \times \{1,\dotsc, r^{-d}\}$ such that
\begin{enumerate}
\item each tile $Q_{i,j}$ has diameter at most $r$,
\item for every $i \in I$, the tiles $(Q_{i,j})_{j \in J}$ are $r$-separated from each other.
\end{enumerate}

\begin{figure*}
\centering
\begin{tikzpicture}[scale=.7]
\foreach \x in {0,...,3}{
\foreach \y in {0,...,3}{
    \path [fill=gray] (2*\x, 2*\y) rectangle (2*\x+1,2*\y+1);
    \pgfmathtruncatemacro{\j}{\x + 4 * \y + 1}
    \node at (2*\x+0.5, 2*\y+0.5) {\tiny $Q_{1, \j}$};
}}
\draw (0, 0) grid (8, 8);
\end{tikzpicture} 
\caption{Illustration of $Q_1$ (the gray area) in the partition for $d=2$.}
\end{figure*}

For every $(i,j) \in I \times J$, let $x_{i,j}$ be a random variables taking value in $Q_{i,j} \subset \Tbb^d$ distributed according to renormalized restriction of $\nu$ to $Q_{i,j}$.
For $i \in I$, define $Q_i = \bigsqcup_{j \in J} Q_{i,j}$ and 
\[\nu_i = \sum_{j\in J} \frac{\nu(Q_{i,j})}{\nu(Q_i)} \delta_{x_{i,j}}.\]
By its definition, for every $i \in I$, $\nu_i$ is a random probability measure on $\Tbb^d$ whose support is almost surely $r$-separated.
Moreover, almost surely,
\begin{equation}
\label{eq:nu'Delta}
\nu_i \otimes \nu_i (\Delta) \leq \max_{y \in \Tbb^d} \nu_i(y)
 \leq \max_{j \in J} \frac{\nu(Q_{i,j})}{\nu(Q_i)} \leq \frac{s}{\nu(Q_i)}.
\end{equation}

Finally, for any bounded measurable function $f \colon \Tbb^d \to \R$, we have
\[\int_{\Tbb^d} f \dd \nu = \sum_{i \in I} \nu(Q_i) \Ebb\Bigl[\int_{\Tbb^d} f \dd \nu_i \Bigr],\]
where $\Ebb$ denote the expectation (remember that $\nu_i$ are random).
By  the pigeonhole principle, there exists $i \in I$ such that
\[\nu(Q_i) \Ebb\Bigl[\int_{\Tbb^d} f \dd \nu_i \Bigr] \geq 2^{-d} \int_{\Tbb^d} f \dd \nu.\]
If $f$ takes value in ${[0,1]}$, for this $i$,
\[\text{both } \nu(Q_i) \text{ and } \Ebb\Bigl[\int_{\Tbb^d} f \dd \nu_i \Bigr] \geq 2^{-d} \int_{\Tbb^d} f \dd \nu.\]
This together with~\eqref{eq:nu'Delta} finishes the proof of the lemma.
\end{proof}

Using Lemma~\ref{lm:decomp}, we can show that \eqref{eq:initNC} and \eqref{eq:lateC} are not compatible, which will complete the proof of Theorem~\ref{thm:affineTd}.

Consider the function $f \colon \Tbb^d \to {[0,1]}$,
\[z \mapsto \int_{\SL_d(\Z)\ltimes \Tbb^d} \indic_{B(y,\rho)}(gz) \dd \mu^{*n_2}(g),\]
so that $\int_{\Tbb^d}f \dd \nu = \mu^{*n_2}*\nu(B(y,\rho))$.
If follows from Lemma~\ref{lm:decomp}, \eqref{eq:initNC} and \eqref{eq:lateC} that there exists a $r$-separated probability measure $\nu'$ on $\Tbb^d$ such that
\[\mu^{*n_2}*\nu' (B(y,\rho)) \gg \rho^\beta \quad \text{and} \quad \nu'\otimes \nu'(\Delta) \ll e^{-\frac{n_1}{C_1}}\rho^{-\beta}.\]
Since $\nu'$ is $r$-separated, we have 
\[\Ecal_\alpha(\nu') \leq r^{-\alpha}.\]
Thus, applying Proposition~\ref{pr:MargulisF} to the parameter $\lambda'$ and the measure $\nu'$ and remembering the choice of $\beta$ and $r$, we obtain
\[ \rho^{2\beta} \ll e^{-\frac{n_1}{C_1}} \rho^{-\beta} + \rho^{3\beta},\]
where the implied constant depends only on $\gamma$, $\Pbb$ and $\lambda$. On account of \eqref{eq:rangerho}, this leads to a contradiction provided that we choose $L_1 = 4C_1 C_3 L_3$ and $L_3$ to be a large multiple of $C_3$ large enough so that the left hand side divided by the right hand side is greater than the implicit constant in the $\ll$ notation above.
\end{proof}

\begin{newchanges}
The endgame strategy for the proof of Theorem~\ref{thm:qualitative} given in this section is somewhat different than the one used for the linear random walk by Bourgain, Furman, Mozes and the third named author in \cite{BFLM}, which is essentially followed by the first named author in \cite{He_schubert} as well as in the joint work with de Saxc\'e in~\cite{HS} (e.g.\ in \cite{BFLM} this endgame is the content of \cite[Prop.~7.3 and~7.4]{BFLM}). It is possible to apply a similar endgame strategy in which a Margulis function argument similar to that we employ in \S\ref{sc:n2} (together with a much simplified variant of Proposition~\ref{pr:initNC}) is used as a substitute to these two propositions in \cite{BFLM}, giving an alternative treatment to that portion of the argument, that is arguably a bit simpler especially in the non-proximal cases handled in \cite{He_schubert, HS}. 
\end{newchanges}

\subsection{Qualitative statement}
As announced in the introduction, the qualitative statement Theorem~\ref{thm:qualitative} can be deduced from Theorem~\ref{thm:affineTd}.
\begin{proof}[Proof of Theorem~\ref{thm:qualitative}]
We can realize $\mu$ as $\mu = (\gamma,u)_* \Pbb$ like in the statement of Theorem~\ref{thm:affineTd}.

Assume that the sequence $\mu^{*n} * \delta_x$ does not converge to the Haar measure on $\Tbb^d$.
Then by Weyl's criterion, there exists $a \in \Z^d \setminus \{0\}$ and $t > 0$ such that 
\(\abs{\widehat{\mu^{*n}*\delta_x}(a)} \geq t\)
for arbitrary large $n$.
By Theorem~\ref{thm:affineTd}, \( d((u,x) , \Pcal_Q) \leq e^{-cn}\) for $Q = \bigl( \frac{\norm{a}}{t} \bigr)^C$. The set $\Pcal_Q$ being closed, with $n$ goes to $+\infty$, we conclude that $(u,x) \in \Pcal_Q$.
\end{proof}

\subsection{Optimality in the convergence rate}
Finally, we prove Proposition~\ref{pr:rate}.
\begin{proof}[Proof of Proposition~\ref{pr:rate}]
For each $\omega \in \Omega$, consider the linear map $\tilde{\gamma}(\omega) \colon (\R^d)^\Omega \times \R^d \to (\R^d)^\Omega \times \R^d$ defined by
\[\tilde{\gamma}(\omega)(u,x) = (u, \gamma(\omega)x + u(\omega)).\]
For each $n\geq 1$, we can extend the definition of $\tilde{\gamma} \colon \Omega \to \GL\bigl((\R^d)^\Omega \times \R^d\bigr)$ to $\Omega^n$ by setting for every $\newchange{\omegatuple} = (\omega_n,\dotsc, \omega_1) \in \Omega^n$,
\[\tilde{\gamma}(\newchange{\omegatuple}) = \tilde{\gamma}(\omega_n) \dotsm \tilde{\gamma}(\omega_1) \in \GL\bigl((\R^d)^\Omega \times \R^d\bigr).\]
Then for any $(u,x) \in (\R^d)^\Omega \times \R^d$, we have 
\[\forall n \geq 1,\, \forall \newchange{\omegatuple}  \in \Omega^n,\quad \tilde{\gamma}(\newchange{\omegatuple})(u,x) = \bigl(u, (\gamma,u)(\newchange{\omegatuple})x \bigr).\]

The push-forward measure $\tilde{\gamma}_*\Pbb$ defines a linear random walk on $(\R^d)^\Omega \times \R^d$.
Note that, $\forall \omega \in \Omega$, $\tilde{\gamma}(\omega)$ is triangular by block, with diagonal blocks being $1$'s and $\gamma(\omega)$. Hence by a result of Furstenberg and Kifer~\cite[Lemma 3.6]{FurstenbergKifer}, we have a equality of the top Lyapunov exponents, (recalling $\mu_0 = \gamma_*\Pbb$), 
\[\lambda_{1,\tilde{\gamma}_*\Pbb} = \lambda_{1,\mu_0}.\]
\newchange{We would like to use a large deviation estimate for the norm of random matrix products~\cite[Theorem V.6.2]{BougerolLacroix}.
More precisely, we need Boyer's version~\cite[Theorem A.5]{Boyer} based on the approach in~\cite{BQ_CLT}, which is valid for non-irreducible actions.}
\newchange{We obtain,} given $\epsilon > 0$, there exists $c  = c( \gamma_*\Pbb,\epsilon) > 0$ such that \newchange{for all $n$ sufficiently large,}
\[\Pbb^{\otimes n} \bigl\{\, \newchange{\omegatuple}  \in \Omega^n \mid \norm{\tilde{\gamma}(\newchange{\omegatuple} )} \leq e^{(\lambda_{1,\mu_0} + \epsilon) n} \,\bigr\} \geq 1 - e^{-cn}.\]

Recall the assumption $\lambda > \lambda_{1,\mu_0}$. Set $\epsilon = \frac{\lambda - \lambda_{1,\mu_0}}{2}$.
If $(u,x), (v,y) \in (\R^d)^\Omega \times \R^d$ satisfy $\norm{(u,x) - (v,y)} \leq e^{-\lambda n}$, then
\[\Pbb^{\otimes n} \{\,\newchange{\omegatuple}  \in \Omega^n \mid \norm{(\gamma,u)(\newchange{\omegatuple})x - (\gamma,v)(\newchange{\omegatuple})y} \leq e^{-\epsilon n} \,\} \geq 1 - e^{-cn}.\]
If moreover $(v,y)$ projects to an element in $\Pcal_Q \subset (\Tbb^d)^\Omega \times \Tbb^d$, i.e. \newchange{there exists $q \leq Q$ such that} for all $n \geq 1$,
$(\gamma,v)(\Omega^n)y \subset y + \frac{1}{q}\Z^d$
\newchange{, then}
\[\Pbb^{\otimes n} \bigl\{\,\newchange{\omegatuple} \in \Omega^n \mid (\gamma,u)(\newchange{\omegatuple})x \in y + \frac{1}{q}\Z^d + B(0,e^{-\epsilon n}) \,\bigr\} \geq 1 - e^{-cn}.\]

Observe that for all $a \in q \Z^d$ and all $z \in y + \frac{1}{q}\Z^d + B(0,e^{-\epsilon n})$,
\[\bigabs{e^{2\pi i\langle a, z\rangle}  - e^{2 \pi i \langle a, y\rangle}} \ll \norm{a} e^{-\epsilon n}.\]
If follows that
\[ \bigabs{\widehat{\mu^{*n}*\delta_{\pi(x)}}(a)  - e^{2\pi i\langle a, y\rangle} } \ll \norm{a}e^{-\epsilon n} + e^{-c n},\]
where $\mu = (\gamma,u)_*\Pbb$ and $\pi \colon \R^d \to \Tbb^d$ denotes the canonical projection.
This finishes the proof of the proposition.
\end{proof}
\appendix
\section{Affine random walks on linear spaces over prime fields}
\label{appendix}
The goal of this appendix is to prove Proposition~\ref{pr:affineFpd}. 
We first establish the following variant, 
then deduce Proposition~\ref{pr:affineFpd} from it.
\begin{prop}
\label{pr:gapFpd}
Let $\mu_0$ be a probability measure on $\SL_d(\Z)$.
Let $\Gamma$ denote the subgroup generated by $ \Supp(\mu_0)$.
Assume that the action of $\Gamma$ on $\Q^d$ is strongly irreducible and that its Zariski closure is semisimple.
Given $\epsilon > 0$,
there exists $C = C(\mu_0,\epsilon) > 0$ and $p_0 = p_0(\mu_0)$ such that the following holds.

For any prime number $p \geq p_0$ and any probability measure $\mu$ on $\SL_d(\Fbb_p) \ltimes \Fbb_p^d$. If the projection of $\mu$ to $\SL_d(\Fbb_p)$ is the reduction of $\mu_0$ modulo $p$, then 
either
there exists $x,y \in \Fbb_p^d$ such that
\[ \mu ( \{ g \in \SL_d(\Fbb_p) \ltimes \Fbb_p^d \mid g x = y \} ) \geq 1 - \epsilon\]
or for any $x, y \in \Fbb_p^d$,  
\[\forall n \geq 1 , \quad \mu^{*n} ( \{ g \in \SL_d(\Fbb_p) \ltimes \Fbb_p^d \mid g x = y \} ) \leq C \max\{p^{-\frac{1}{4}}, e^{-n/C}\}.\]
\end{prop}

To prove this proposition we follow closely the arguments in Lindenstrauss-Varj\'u~\cite{LV}.

\subsection{A non-concentration estimate}
In this subsection, we fix a prime number $p$ and consider affine random walks on $\Fbb_p^d$.
The aim is to establish a non-concentration in logarithmic time, provided that there is no fixed point and there is a spectral gap for the associated linear random walk.
This will be a variant of \cite[Proposition 3]{LV}.

Let $\Gamma$ be a subgroup of $\SL_d(\Fbb_p)$.
Let $\Lcal^\theta \colon \Gamma \to \GL(L^2(\Gamma))$ be the left regular representation of $\Gamma$.
Let $\Lcal_0^\theta$ denote the restriction of $\Lcal^\theta$ to $L^2_0(\Gamma)$, the space of mean zero functions.

The space of functions on $\Fbb_p^d$ is equipped \newchange{with} the usual $L^q$-norm for exponent $q \geq 1$, i.e. for any $f \in L^q(\Fbb_p^d)$,
\[\norm{\phi}_{L^q}^q = \sum_{x \in \Fbb_p^d} \abs{f(x)}^q.\]
We will abbreviate $L^q(\Fbb_p^d)$ simply by $L^q$.

Let $\Acal$ be the Koopman representation associated to the action of $\Gamma \ltimes \Fbb_p^d$ on $\Fbb_p^d$, i.e. the unitary representation $L^2$ of $\Gamma \ltimes \Fbb_p^d$ defined by
\[\forall g \in \Gamma \ltimes \Fbb_p^d,\, \forall f \in L^2,\, \forall x \in \Fbb_p^d,\quad \Acal(g) f (x) = f(g^{-1} x).\]

Whenever we have a linear representation, we extend by linearity the morphism to the group algebra, e.g. for a measure $\mu$ on $\Gamma \ltimes \Fbb_p^d$ and a measure $\eta$ on $\Fbb_p^d$ (viewed as a function), we have
\[\Acal(\mu)\eta = \mu * \eta.\]

\begin{prop}
\label{pr:LV}
Let $\Gamma \subset \SL_d(\Fbb_p)$ be a subgroup.
Let $\mu$ be a probability measure on $\Gamma \ltimes \Fbb_p^d$ and let $\mu_0$ denote its projection to $\Gamma$.
Assume
\begin{enumerate}
\item the only $\Gamma$-orbit in $\Fbb_p^d$ of cardinality less than $p$ is the singleton $\{0\}$;
\item for every $x \in \Fbb_p^d$, $\norm{\Acal(\mu)\delta_x}_{L^2} \leq \frac{3}{4}$;
\item \label{it:thetagap} $\norm{\Lcal_0^\theta(\mu_0)} \leq 2^{-5}$.
\end{enumerate}
then for any integer $l \geq 1$, we have for every $x \in \Fbb_p^d$,
\[\bignorm{\Acal(\mu)^l \delta_x}_{L^\infty} \leq \bignorm{\Acal(\mu)^l \delta_x}_{L^2} \leq \max\{ 19p^{-\frac{1}{4}}, e^{-2^{-35} l}\}.\]
\end{prop}
In condition~\ref{it:thetagap}, the norm is the operator norm with respect to the Hilbert norm on $L^2_0(\Gamma)$. Thus, \ref{it:thetagap} is a spectral gap condition for the linear part of $\mu$.

Compared with \cite[Proposition 3]{LV}, we allow $\Gamma$ to be smaller than $\SL_d(\Fbb^d)$ and instead require that it has no small orbit except the trivial one. 
The proof in  \cite{LV} works after minor modifications. We will focus on explaining the necessary modification and refer the reader to the original article for more details.

The dual ${\hat\Fbb}_p^d$ of $\Fbb_p^d$ is isomorphic to $\Fbb_p^d$, but we make a distinction in our notation so as to have different normalization to the various norms used.
Let $\langle \mybullet, \mybullet \rangle \colon \hat{\Fbb}_p^d \times \Fbb_p^d \to \Fbb_p$ be the usual pairing and $e \colon \Fbb_p \to \C^\times$ a fixed nontrivial character, e.g. $e(t)= e^{\frac{2 \pi i t}{p}}$, $\forall t \in \Fbb_p$.
The (discrete) Fourier transform of a function $f \colon \Fbb_p^d \to \C$ is $\hat{f} \colon \hat{\Fbb}_p^d \to \C$ defined by
\[\forall a \in \hat{\Fbb}_p^d,\quad \hat{f}(a) = \sum_{x \in \Fbb_p^d} e(\langle a, x\rangle) f(x).\]

For exponent $q \geq 1$, we denote by $\hat{L}^q$ the space of functions on ${\hat\Fbb}_p^d$ equipped with the following norm,
\[ \norm{\phi}_{\hat{L}^q}^q = \frac{1}{p^d}\sum_{a \in {\hat\Fbb}_p^d} \abs{\phi(a)}^q.\]
The normalization is to make the discrete Fourier transform into an isometry for $p=2$, i.e.\ so that for any $f \in L^2$ we have that $\norm{f}_{L^2} = \norm{\hat{f}}_{\hat{L}^2}$.

Thus, conjugating the representation $\Acal$ by the Fourier transform we obtain an unitarily equivalent representation $\hat\Acal \colon \Gamma \ltimes \Fbb_p^d \to \GL(\hat{L}^2)$. Explicitly,
\begin{equation}
\label{eq:hatA}
\forall (\gamma,u) \in \Gamma \ltimes \Fbb_p^d,\,  \forall \phi \in \hat{L}^2,\, \forall a \in \hat\Fbb_p^d,\quad \hat{\Acal}(\gamma,u) \phi (a) = e(\langle a,u \rangle)\phi(\gamma^{\tr} a),
\end{equation}
where $\gamma^{\tr} \in \GL(\hat{\Fbb}_p^d)$ denotes the transpose of $\gamma$.

Let $\Acal^\theta \colon \Gamma \to \GL(L^2)$ denote the restriction of $\Acal$ to $\Gamma$. 
Let $\hat{\Acal}^\theta \colon \Gamma \to \GL(\hat{L}^2)$ denote the conjugate of $\Acal^\theta$ by the Fourier transform.
Concretely, it is defined by
\begin{equation}
\label{eq:hatAtheta}
\forall \gamma \in \Gamma,\,  \forall \phi \in \hat{L}^2,\, \forall a \in \hat\Fbb_p^d,\quad \hat{\Acal}^\theta(\gamma) \phi (a) = \phi(\gamma^{\tr} a),
\end{equation}
which makes the following property true,
\begin{equation}
\label{eq:Ahatfhat}
\forall \gamma \in \Gamma,\, \forall f \in L^2,\quad \hat{\Acal}^\theta(\gamma) \hat{f} = (\Acal^\theta(\gamma)f)^{\widehat{}}.
\end{equation}

\begin{lemm}
\label{lm:gapinL4}
Let $\eta$ be a probability measure on $\Fbb_p^d$.
Let $\Gamma \subset \SL_d(\Fbb_p)$ be a subgroup.
Assume
\begin{enumerate}
\item the only $\Gamma$-orbit in $\Fbb_p^d$ of cardinality less than $p$ is the singleton $\{0\}$;
\item \label{it:noBmass} for every $x \in \Fbb_p^d$, \(\eta(x) \leq \frac{40}{41} \norm{\eta}_{L^2}\);
\item \label{it:bigL4norm} \(\norm{\hat{\eta}}_{\hat{L}^4} \geq 19 p^{-\frac{1}{4}}\).
\end{enumerate}
Then there exists $h \in \Gamma$ such that
\[\bignorm{ \abs{\hat{\eta}} -  \hat{\Acal}^\theta(h) \abs{\hat{\eta}}}_{\hat{L}^4} \geq \frac{7}{100} \norm{\hat{\eta}}_{\hat{L}^4}.\]
\end{lemm}
\begin{proof}
Let $\nu = \check{\eta} * \eta$ where $\check{\eta}(x) = \eta(-x)$ and $*$ is the additive convolution. Note that $\nu$ is a probability and 
\[ \hat{\nu} = \abs{\hat{\eta}}^2.\]
Hence, by assumption~\ref{it:bigL4norm},
\[\norm{\nu}_{L^2} = \norm{\hat{\nu}}_{\hat{L}^2} = \norm{\hat{\eta}}^2_{\hat{L}^4} \geq 350 p^{-\frac{1}{2}}.\]

By a property of the Mazur map 
\newchange{(\cite[Theorem C]{LV}\footnote{The proof of \cite[Theorem C]{LV} is contained in \cite[Proof of Theorem 9.1]{BenyaminiLindenstrauss}.} applied with $f_1 = \frac{ \abs{\hat{\eta}}}{\norm{\hat{\eta}}_{\hat{L}^4} }$ and $f_2 = \frac{\Acal^\theta(h)\abs{\hat{\eta}}}{\norm{\hat{\eta}}_{\hat{L}^4}}$)}, the observation \eqref{eq:Ahatfhat} and the \newchange{triangle} inequality, 
\begin{equation}
\label{eq:mazurTriangleIneq}
\begin{aligned}
\frac{1}{\abs{\Gamma}}\sum_{h \in \Gamma} \frac{\bignorm{ \abs{\hat{\eta}} -  \hat{\Acal}^\theta(h) \abs{\hat{\eta}}}_{\hat{L}^4}}{\norm{\hat{\eta}}_{\hat{L}^4} } &\geq 
\frac{1}{2\abs{\Gamma}}\sum_{h \in \Gamma} \frac{\bignorm{ \abs{\hat{\eta}}^2 -  \hat{\Acal}^\theta(h) \abs{\hat{\eta}}^2}_{\hat{L}^2}}{ \newchange{\norm{\hat{\eta}}_{\hat{L}^4}^2} }\\
&\geq \frac{1}{2\abs{\Gamma}}\sum_{h \in \Gamma} \frac{\bignorm{\hat{\nu} -  \hat{\Acal}^\theta(h) \hat{\nu}}_{\hat{L}^2}}{\norm{\nu}_{L^2}}\\
&\geq \frac{1}{2\abs{\Gamma}} \sum_{h \in \Gamma} \frac{\bignorm{\nu -  \Acal^\theta(h) \nu}_{L^2}}{\norm{\nu}_{L^2}}\\
&\geq \frac{\bignorm{\nu -  \Acal^\theta(\mu_\Gamma) \nu}_{L^2}}{2\norm{\nu}_{L^2}},
\end{aligned}
\end{equation}
where $\mu_\Gamma$ \newchange{denotes} the uniform probability measure on $\Gamma$.

\newchange{Note that} $\Acal^\theta(\mu_\Gamma) \nu$ is a convex combination of uniform probability measures on $\Gamma$-orbits. After we remove the contribution of the trivial orbit, $\Acal^\theta(\mu_\Gamma) \nu -  \nu(0) \delta_0$ is supported on orbits of cardinality at least $p$. 
Hence 
\[\bignorm{\Acal^\theta(\mu_\Gamma) \nu -  \nu(0) \delta_0}_{L^\infty} \leq p^{-1}.\]
Consequently,
\[\bignorm{\Acal^\theta(\mu_\Gamma) \nu -  \nu(0) \delta_0}_{L^2}  \leq p^{-\frac{1}{2}} \leq \frac{1}{350}\norm{\nu}_{L^2}.\]
On the other hand, by \cite[\newchange{Lemma 5}]{LV}, using assumption~\ref{it:noBmass}, 
\[\bignorm{\nu - \nu(0)\delta_0}_{L^2} \geq \frac{1}{7}\norm{\nu}_{L^2}.\]
Hence 
\[\frac{\bignorm{\nu -  \Acal^\theta(\mu_\Gamma) \nu}_{L^2}}{2\norm{\nu}_{L^2}} \geq \frac{1}{2}\Bigl(\frac{1}{7} - \frac{1}{350}\Bigr) \geq \frac{7}{100}.\]
Combined with \eqref{eq:mazurTriangleIneq},  this finishes the proof.
\end{proof}

\begin{lemm}
\label{lm:L4decay}
Let $\Gamma \subset \SL_d(\Fbb_p)$ be a subgroup.
Let $\mu_0$ be a probability measure on $\Gamma$ and $\eta$ be a probability measure on $\Fbb_p^d$.
Assume that
\begin{enumerate}
\item the only $\Gamma$-orbit in $\Fbb_p^d$ of cardinality less than $p$ is the singleton $\{0\}$;
\item $\norm{\Lcal_0^\theta(\mu_0)} \leq 2^{-5}$;
\item for every $x \in \Fbb_p^d$, \(\eta(x) \leq \frac{40}{41} \norm{\eta}_{L^2}\);
\item \(\norm{\hat{\eta}}_{\hat{L}^4} \geq 19 p^{-\frac{1}{4}}\).
\end{enumerate}
Then
\[\bignorm{\hat{\Acal}^\theta(\mu_0) \abs{\hat\eta} }_{\hat{L}^4} \leq 2^{-2^{-34}} \norm{\hat\eta}_{\hat{L}^4}.\]
\end{lemm}
\begin{proof}
Let $\check{\mu}_0$ denote the measure $\forall \gamma \in \Gamma$, $\check{\mu}_0(\gamma) = \mu_0(\gamma^{-1})$. We have
\[\bignorm{\Lcal^\theta_0\bigl(\check\mu_0 * \mu_0\bigr)} = \bignorm{\Lcal^\theta_0(\mu_0)^*\, \Lcal^\theta_0( \mu_0)} 
\leq \norm{\Lcal^\theta_0(\mu_0) }^{2} \leq 2^{-10}\]
where $\Lcal^\theta_0(\mu_0)^* $ denotes the adjoint operator of $\Lcal^\theta_0(\mu_0)$.

Note that $\Lcal^\theta_0$ contains every nontrivial unitary irreducible representation of $\Gamma$. Thus, for every nontrivial unitary irreducible representation $\rho$ of $\Gamma$,
\[\bignorm{\rho\bigl(\check\mu_0 * \mu_0\bigr)} \leq 2^{-10},\]
where $\norm{\mybullet}$ denotes the operator norm subordinated to the norm on the Hilbert space of $\rho$.

Decompose $\hat{\Acal}^\theta$ into $\hat{\Acal}^\theta = \hat{\Acal}^\theta_1 \oplus \hat{\Acal}^\theta_0$ where $\hat{\Acal}^\theta_1$ is a sum of trivial representations and $\hat{\Acal}^\theta_0$ is a sum of nontrivial irreducible representations. In this decomposition, for every $h \in \Gamma$,
\[\bigl( \hat{\Acal}^\theta(h)  - 1 \bigr) \hat{\Acal}^\theta\bigl(\check\mu_0 * \mu_0\bigr) = 0 \oplus \bigl( \hat{\Acal}_0^\theta(h)  - 1 \bigr) \hat{\Acal}_0^\theta\bigl(\check\mu_0 * \mu_0\bigr).\]
Thus, this operator has $\hat{L}^2$-operator norm $\leq 2^{-9}$ and $\hat{L}^{\infty}$-operator norm $\leq 2$. By the Riesz–Thorin interpolation theorem,
\begin{equation}
\label{eq:RieszThorin}
\Bignorm{\bigl( \hat{\Acal}^\theta(h)  - 1 \bigr) \hat{\Acal}^\theta\bigl(\check\mu_0 * \mu_0\bigr)}_{\hat{L}^4} \leq 2^{-4}.
\end{equation}

To lighten the notation, write $\phi = \frac{\abs{\hat\eta}}{\norm{\hat\eta}_{\hat{L}^4}}$ so that $\phi \in \hat{L}^4$ is a unit vector.
By Lemma~\ref{lm:gapinL4}, there exists $h \in \Gamma$ such that
\[\bignorm{ \hat{\Acal}^\theta(h)\phi - \phi}_{\hat{L}^4} \geq \frac{7}{100}.\]
We can write $\hat{\Acal}^\theta(h)\phi  - \phi$ as
\[ \hat{\Acal}^\theta(h) \bigl(\phi - \hat{\Acal}^\theta\bigl(\check\mu_0 * \mu_0\bigr)\phi \bigr) - \bigl(\phi - \hat{\Acal}^\theta\bigl(\check\mu_0 * \mu_0\bigr)\phi \bigr) - \bigl( \hat{\Acal}^\theta(h)  - 1 \bigr) \hat{\Acal}^\theta\bigl(\check\mu_0 * \mu_0\bigr)\phi.\]
Hence, by the triangle inequality, the fact that $\hat{\Acal}^\theta(h)$ is an isometry of $\hat{L}^4$ and inequality~\eqref{eq:RieszThorin},
\[ \frac{7}{100} \leq 2 \bignorm{ \hat{\Acal}^\theta\bigl(\check\mu_0 * \mu_0\bigr)\phi - \phi}_{\hat{L}^4} +  2^{-4}.\]
It follows that
\[\bignorm{ \hat{\Acal}^\theta\bigl(\check\mu_0 * \mu_0\bigr)\phi - \phi}_{\hat{L}^4} \geq \frac{3}{800}.\]

Remark again that $\hat{\Acal}^\theta(\gamma)$, $\gamma \in \Gamma$ are isometries of $\hat{L}^4$.
By the triangle inequality,
\begin{align*}
\bignorm{ \hat{\Acal}^\theta\bigl(\check\mu_0 * \mu_0\bigr)\phi - \phi}_{\hat{L}^4} &\leq \sum_{\gamma, \gamma' \in \Gamma} \mu_0(\gamma)  \mu_0(\gamma') \bignorm{\hat{\Acal}^\theta(\gamma^{-1}) \hat{\Acal}^\theta(\gamma') \phi - \phi}_{\hat{L}^4}\\
&= \sum_{\gamma, \gamma' \in \Gamma} \mu_0(\gamma)  \mu_0(\gamma') \bignorm{ \hat{\Acal}^\theta(\gamma') \phi -  \hat{\Acal}^\theta(\gamma)\phi}_{\hat{L}^4}.
\end{align*}
Also by the triangle inequality,
\[\bignorm{ \hat{\Acal}^\theta(\mu_0) \phi}_{\hat{L}^4} \leq  \sum_{\gamma, \gamma' \in \Gamma} \mu_0(\gamma)  \mu_0(\gamma') \Bignorm{ \frac{\hat{\Acal}^\theta(\gamma') \phi +  \hat{\Acal}^\theta(\gamma)\phi}{2}}_{\hat{L}^4}\]
By \cite[Lemma 7]{LV}, for every $\gamma, \gamma' \in \Gamma$, 
\begin{newchanges}
\[1 - \frac{7}{16} \bignorm{ \hat{\Acal}^\theta(\gamma') \phi -  \hat{\Acal}^\theta(\gamma)\phi}_{\hat{L}^4}^4 \geq 0\]
and
\end{newchanges}
\[\Bignorm{ \frac{\hat{\Acal}^\theta(\gamma') \phi +  \hat{\Acal}^\theta(\gamma)\phi}{2}}_{\hat{L}^4} \leq \sqrt[4]{1 - \frac{7}{16} \bignorm{ \hat{\Acal}^\theta(\gamma') \phi -  \hat{\Acal}^\theta(\gamma)\phi}_{\hat{L}^4}^4}.\]
Finally, by Jensen's inequality (applied to \newchange{the concave function} $t \mapsto \sqrt[4]{1 - \frac{7}{16}t^4}$), we obtain, 
\[\bignorm{ \hat{\Acal}^\theta(\mu_0) \phi}_{\hat{L}^4} \leq \sqrt[4]{1 - \frac{7}{16}\frac{3^4}{800^4}} \leq e^{-2^{-34}}.\qedhere\]
\end{proof}

\begin{proof}[Proof of Proposition~\ref{pr:LV}]
Given a probability measure $\eta$ on $\Fbb_p^d$, we claim that
\begin{equation}
\label{eq:L2orL4decay}
\text{either}\ \norm{\hat\eta}_{\hat{L}^4} < 19 p^{-\frac{1}{4}} \ \text{or} \ 
\bignorm{\Acal(\mu)\eta}_{L^2} \bignorm{\widehat{\Acal(\mu)\eta}}_{\hat{L}^4} \leq e^{-2^{-34}} \norm{\eta}_{L^2} \norm{\hat{\eta}}_{\hat{L}^4}. 
\end{equation}

On the one hand, if there exists $x_0 \in \Fbb_p^d$ such that \(\eta(x_0) > \frac{40}{41}\norm{\eta}_{L^2}\), then
\begin{align*}
\bignorm{\Acal(\mu)\eta}_{L^2} &\leq \eta(x_0) \bignorm{\Acal(\mu)\delta_{x_0}}_{L^2} + \norm{\eta - \eta(x_0) \delta_{x_0}}_{L^2} \\
& \leq \frac{3}{4} \eta(x_0) + \sqrt{\newchange{\norm{\eta}_{L^2}^2} - \eta(x_0)^2}\\
& \leq e^{-2^{-6}} \norm{\eta}_{L^2}.
\end{align*}

Otherwise, we have for every $x \in \Fbb_p^d$, \(\eta(x) \leq \frac{40}{41}\norm{\eta}_{L^2}\).
Observe that, from \eqref{eq:hatA} and \eqref{eq:hatAtheta}, we have
\begin{equation*}
\forall \phi \in \hat{L}^2,\quad \bigabs{\hat\Acal(\mu) \phi} \leq \hat\Acal^\theta(\mu_0) \abs{\phi} \quad \text{pointwise}.
\end{equation*}
Thus, by Lemma~\ref{lm:L4decay}, either $\norm{\hat\eta}_{\hat{L}^4} < 19 p^{-\frac{1}{4}}$, or
\[\bignorm{\widehat{\Acal(\mu)\eta}}_{\hat{L}^4}  = \bignorm{\hat\Acal(\mu) \hat\eta}_{\hat{L}^4} \leq \bignorm{\hat\Acal^\theta(\mu_0) \abs{\hat\eta}}_{\hat{L}^4} \leq e^{-2^{-34}} \norm{\hat\eta}_{\hat{L}^4}.\]
This proves the claim~\eqref{eq:L2orL4decay}.

Applying the claim to $\Acal(\mu)^k \eta$ for every $k = 0,\dotsc, l-1$, we obtain
\[\text{either}\ \Bignorm{\widehat{\Acal(\mu)^l\eta}}_{\hat{L}^4} < 19 p^{-\frac{1}{4}} \ \text{or} \ 
\bignorm{\Acal(\mu)^l\eta}_{L^2} \Bignorm{\widehat{\Acal(\mu)^l\eta}}_{\hat{L}^4} \leq e^{-2^{-34}l} \norm{\eta}_{L^2} \norm{\hat{\eta}}_{\hat{L}^4}.\]
We get the desired estimate by \newchange{taking $\eta = \delta_x$ and recalling that
}
\[\bignorm{\Acal(\mu)^l\eta}_{L^2} = \Bignorm{\widehat{\Acal(\mu)^l\eta}}_{\hat{L}^2} \leq \Bignorm{\widehat{\Acal(\mu)^l\eta}}_{\hat{L}^4}.\qedhere\]
\end{proof}

\subsection{Proof of Proposition~\ref{pr:gapFpd}.}
Now we are going to prove Proposition~\ref{pr:gapFpd}. Basically, we would like to use Proposition~\ref{pr:LV}. Let us check the three assumptions of Proposition~\ref{pr:LV}.

First we have a lower bound on the size of $\Gamma$-orbits in $\Fbb_p^d\setminus\{0\}$.
\begin{lemm}
\label{lm:largeOrbits}
\newchange{Given $d\geq 2$, there is a constant $C = C(d)$ such that the following holds.
Let $\Gamma$ be a subgroup of $\SL_d(\Z)$.
Assume that $\Gamma$ acts irreducibly on $\Q^d$ and the cardinality of $\Gamma$ is at least $C$.}
Then for all but finitely many primes $p$, the only $\Gamma$-orbit in $\Fbb_p^d$ of \newchange{cardinality} less than $p$ is the singleton $\{0\}$.
\end{lemm}
\newchange{Remark that the assumption is satisfied particularly when $\Gamma$ acts strongly irreducibly on $\Q^d$.}
\begin{proof}
\newchange{It is not difficult to see that $\SL_d(\Z)$ contains a torsion-free subgroup of finite index (e.g. the principal congruence subgroup of level $2d+1$). Let  $C = C(d)$ be the index of such a subgroup so that $\SL_d(\Z)$ contains no torsion subgroup of order greater than $C$.

Replacing $\Gamma$ by a subgroup if necessary, we may assume without loss of generality, that $\Gamma$ is finitely generated, acts irreducibly on $\Q^d$ and has order greater than $C$ and hence is not a torsion group.}

Let $\pi_p \colon \SL_d(\Z) \to \SL_d(\Fbb_p)$ denote the reduction modulo $p$ map.
Let $G$ be the Zariski closure of $\Gamma$ in $(\SL_d)_\Z$.
It is a group scheme over $\Z$. Its $\Fbb_p$-points $G(\Fbb_p)$ form a subgroup of $\SL_d(\Fbb_p)$.
By Nori's strong approximation theorem~\cite[Theorem 5.1]{Nori}, we have for all sufficiently large primes $p$,
\[G(\Fbb_p)^+ \subset \pi_p(\Gamma) \subset G(\Fbb_p),\]
where $G(\Fbb_p)^+$ denote the subgroup of $G(\Fbb_p)$ generated by its elements of order $p$, (recall that, when $p \geq d$, these elements are exactly the unipotent elements).
Moreover, by \cite[Remark 3.6]{Nori}, $G(\Fbb_p)^+$ has \newchange{index} at most $2^{d-1}$ in $G(\Fbb_p)$.
Put $l = (2^{d-1})!$, so that given a nonzero point $x \in \Fbb_p^d$, 
either 
\begin{enumerate}
\item \label{it:big orbit} $\pi_p(\Gamma)$ contains an element $g$ of order $p$ such that $gx \neq x$,  or
\item \label{it:xfixed+} $x$ is fixed by $G(\Fbb_p)^+$, and hence for all $\gamma \in \Gamma$, $\pi_p(\gamma)^l x = x$.
\end{enumerate}
If \ref{it:big orbit} holds then $\abs{\Gamma x} \geq \abs{ \langle g \rangle x} \geq p$ which is what we want to establish. Hence it remains to show that the second case \ref{it:xfixed+} can happen only for finitely many primes.

First, we claim that the system of linear equations
\begin{equation}
\label{eq:gammalyy}
\gamma^l y - y = 0,\quad   \gamma \in \Gamma
\end{equation}
does not have nonzero solution in $y \in \Q^d$. 
Indeed, it is easy to see that the space of solutions over $\Q$ to the system of equations \eqref{eq:gammalyy} is $\Gamma$-invariant, hence by \newchange{irreducibility} if there is even one non-zero solution this space has to equal to $\Q^d$. This implies that for any $\gamma \in \Gamma$, \ $\gamma^l = 1$, \newchange{contradicting the assumption that $\Gamma$ is not torsion}.

Thus, we can extract from the system~\eqref{eq:gammalyy} a subsystem consisting of $d$ equations and of nonzero determinant. Let $D$ denote the determinant. It is an integer depending only on $\Gamma$. If $p$ does not divide $D$, then the reduction modulo $p$ of the system~\eqref{eq:gammalyy} does not admit nonzero solution in $\Fbb_p^d$, i.e. the case~\ref{it:xfixed+} does not happen.
\end{proof}

Then we need an initial decay in $L^2$. We use the notation $\Acal$ introduced in the last subsection.
\begin{lemm}
\label{lm:initDecay}
For any probability measure $\mu$ on $\SL_d(\Fbb_p) \ltimes \Fbb_p^d$ and any $\epsilon > 0$, if 
\[ \max_{x \in \Fbb_p}\, \norm{\Acal(\mu)  \delta_x}_{L^\infty}  \leq 1 - \epsilon.\]
Then for any integer $k\geq 1$, 
\[ \max_{x \in \Fbb_p}\, \norm{\Acal(\mu) ^{k} \delta_x}_{L^\infty}  \leq \frac{1}{2} +(1 - \epsilon)^k.\]
In particular, if $k \geq 3\epsilon^{-1}$,
\[\max_{x \in \Fbb_p}\, \norm{\Acal(\mu)^{k} \delta_x}_{L^2}  \leq \frac{3}{4}.\]
\end{lemm}
\begin{proof}
This is contained in the proof of \cite[Lemma 11]{LV}.
\end{proof}

To check the third assumption in Proposition~\ref{pr:LV}, we need the expansion in perfect groups due to Salehi-Golsefidy and Varj\'u.

Let $\mu_0$ be a probability measure on $\SL_d(\Z)$. Let $\Gamma$ denote the subgroup generated by the support of $\mu_0$. 
For a prime number $p$, let $\Gamma_p$ denote the congruence subgroup 
\[\Gamma_p = \{\,\gamma \in \Gamma \mid  \gamma \equiv 1 \mod p \,\} = \Gamma \cap \ker \pi_p,\]
where $\pi_p \colon \SL_d(\Z) \to \SL_d(\Fbb_p)$ is the reduction modulo $p$ map.
Let $\lambda_{\Gamma/\Gamma_p}$ denote the the quasi-regular representation of $\Gamma$ associated to the subgroup $\Gamma_p$.
Finally let $\lambda_{\Gamma/\Gamma_p}^0$ be the subrepresentation of $\lambda_{\Gamma/\Gamma_p}$ obtained by restricting to the space of zero mean functions.

We will use the aforementioned expansion result in the following form.
\begin{thm}[Salehi-Golsefidy-Varj\'u~{\cite[Theorem 1]{SGV}}]
\label{thm:SGV}
Assume that the Zariski closure of the group $\Gamma$ is semisimple. 
Then there exists $c = c(\mu_0) > 0$ such that for all but finitely many prime numbers $p$,
\[\norm{\lambda_{\Gamma/\Gamma_p}^0(\mu_0)} \leq 1 - c.\]
\end{thm}
In~\cite{SGV} this theorem is stated for the case where $\mu_0$ is the uniform probability measure on a finite symmetric generating set. The general case follows easily as it is explained in \cite[\S 3.1]{HS}.

\begin{proof}[Proof of Proposition~\ref{pr:gapFpd}]
Let $p$ be a prime large enough so that the conclusion of Lemma~\ref{lm:largeOrbits} holds.

By Lemma~\ref{lm:initDecay}, either there exists $x$ and $y$ such that
\[ \mu(\{\,g \in \SL_d(\Fbb_p) \ltimes \Fbb_p^d \mid gx = y \,\}) \geq 1 - \epsilon,\]
in which case there is nothing to prove, or
\[\max_{x \in \Fbb_p}\, \norm{\Acal(\mu^{*k}) \delta_x}_{L^2}  \leq \frac{3}{4},\]
for $k\geq 3{\epsilon}^{-1}$.
Assume that we are in the latter case.

Let $c = c(\mu_0) > 0$ be the constant from Theorem~\ref{thm:SGV}. Then for $k \geq 5c^{-1}$,
\[\norm{\lambda_{\Gamma/\Gamma_p}^0(\mu_0^{*k})} \leq 2^{-5}.\]
Note that, using the notation $\Lcal^\theta_0$ from the last subsection,
\[\lambda_{\Gamma/\Gamma_p}^0 = \Lcal^\theta_0 \circ \pi_p.\]
Thus, the assumptions of Proposition~\ref{pr:LV} are satisfied for $\mu^{*k}$, with 
\[k = \bigl\lceil \max\{3{\epsilon}^{-1}, 5c^{-1}\} \bigr\rceil.\] 
We conclude that, for all $l \geq 0$
\[\max_{x \in \Fbb_p^d}\, \bignorm{\mu^{kl} * \delta_x}_{L^\infty} \leq \max\{ 19p^{-\frac{1}{4}}, e^{-2^{-35} l}\}.\]
Hence for all $n \geq 1$,
\[\max_{x \in \Fbb_p^d}\, \bignorm{\mu^{n} * \delta_x}_{L^\infty} \ll \max\{p^{-\frac{1}{4}}, e^{-n/C}\},\]
where $C = 2^{35}k$.
\end{proof}

\subsection{Proof of Proposition~\ref{pr:affineFpd}}
\label{proof of prop. 2.3}
Proposition~\ref{pr:affineFpd} is a strengthening of  Proposition~\ref{pr:gapFpd} in that it tells us what happens when the group generated by $\Supp(\mu)$ has a fixed point: either the starting point is the fixed point in which case the random walk does not leave this point, or the starting point is not the fixed point in which case we still get a exponential decay in $L^\infty$-norm up to time $\log(p)$.

First, we need a lemma.
\begin{lemm}
\label{lm:fixpoint}
Let $S$ be a finite subset of $\SL_d(\Z)$ of cardinality at least $2$. Assume that $S$ preserves no nontrivial proper subspace in $\Q^d$. Then for every large enough prime $p$, for all $x, y \in \Fbb_p^d$, if for every $g \in \Pi^d S$, $gx = y$ then $x = y = 0$.
\end{lemm}

Here, $\Pi^d S$ denotes the set of products of $d$ elements of $S$.

\begin{proof}
Consider the system of linear equations in $(x,y)$
\[gx - y = 0, \quad g \in \Pi^d S,\]
which makes sense over $\Q$ and over $\Fbb_p$.
If the system has full rank over $\Q$, then it has full rank over $\Fbb_p$ for all large enough primes $p$.
Thus it suffices to show that this system admits no nonzero solution over $\Q$.

Assume for a contradiction that $(x,y) \in \Q^d \times \Q^d$ is a nonzero solution, i.e. $x \neq 0$ and $(\Pi^d S) x = \{y\}$.  
It follows that for every $k = 1,\dotsc, d$, the set $(\Pi^k S) x$ is a singleton. Define $x_0 = x$ and $x_k \in \Q^d$ to be such that 
\begin{equation}
\label{eq:Sxk}
\forall k = 0,\dotsc,d-1,\quad \forall g \in S,\quad gx_k = x_{k+1}.
\end{equation}
The vectors $(x_0,\dotsc,x_d)$ can not be linearly independent. Hence there exists $k \in \{ 1, \dotsc, d \}$ such that $x_k \in \Span(x_0,\dotsc,x_{k-1})$. If follows that $S$ preserves the nonzero subspace $\Span(x_0,\dotsc,x_{k-1})$. Hence $k=d$ and $(x_0,\dotsc,x_{d-1})$ is a basis of $\Q^d$. In view of \eqref{eq:Sxk}, this implies that $S$ is a singleton, which contradicts our assumption.
\end{proof}

\begin{proof}[Proof of {Proposition~\ref{pr:affineFpd}}]
Let $(\Omega,\Pbb), \gamma, u$ be as in the statement of the proposition. Let $S = \gamma(\Omega)$.
Observe that the subgroup generated by the elements of the \newchange{product-set} $\Pi^{d+1} S$ has finite index in the subgroup generated by $S$. Therefore, we can apply Proposition~\ref{pr:gapFpd} to $\mu_0 = (\gamma_*\Pbb)^{*(d+1)}$ with $\epsilon = \min_{\omega \in \Omega} \Pbb(\omega)^{d+1}$.
We obtain that for any $u \colon \Omega \to \Fbb_p^d$,
either 
\[\forall n \geq 1 , \quad \max_{x,y \in \Pbb_p^d} \Pbb^{\otimes (d+1)n} \bigl( \bigl\{\, \newchange{\omegatuple \in \Omega^{(d+1)n}} \mid  (\gamma,u)(\omegatuple) x = y \,\bigr\} \bigr)  \leq C \max\{p^{-1/4},e^{-\frac{n}{C}}\}\]
for some $C = C(\mu_0,\epsilon)$,
in which case we are done,
or there exist $x_0,y_0 \in \Fbb_p^d$ such that
\begin{equation}
\label{eq:x0y0}
\forall \omegatuple \in \Omega^{d+1},\quad (\gamma,u)(\omegatuple)x_0 = y_0.
\end{equation}

In the latter case, we claim that $x_0$ must be a fixed point of $(\gamma,u)(\Omega)$, provided that $p$ is large enough.
Indeed, if follows from~\eqref{eq:x0y0} that there are $x_1, x_d, x_{d+1} \in \Fbb_p^d$ such that
\[\forall \omega \in \Omega,\quad (\gamma,u)(\omega)x_0 = x_1\]
and
\[\forall \omegatuple \in \Omega^d,\quad (\gamma,u)(\omegatuple)x_0 = x_d \text{ and }(\gamma,u)(\omegatuple)x_1 = x_{d+1}.\]
Subtracting the last two equalities we obtain,
\[\forall g \in \Pi^d S,\quad g(x_1 - x_0) = x_{d+1} - x_d.\]
By Lemma~\ref{lm:fixpoint}, if $p$ is sufficiently large,  $x_0 = x_1$ proving the claim. Moreover, by Lemma~\ref{lm:fixpoint} again, we know that this fixed point is unique.

It remains to prove that for $x \in \Fbb_p^d \setminus \{x_0 \}$ and $y \in \Fbb_p^d$,
\begin{equation}
\label{eq:gaponFpd}
\forall n \geq 1 , \quad \Pbb^{\otimes n} \bigl( \bigl\{\, \newchange{\omegatuple \in \Omega^n} \mid  (\gamma,u)(\newchange{\omegatuple}) x = y \,\bigr\} \bigr) \leq C \max\{p^{-1/4},e^{-\frac{n}{C}}\}
\end{equation}
for some $C = C(\gamma_*\Pbb)$.
Indeed, after conjugating by the translation by $x_0$, we may assume $x_0 = 0$ and $u = 0$, i.e. the random walk on $\Fbb_p^d$ is linear, induced by the reduction modulo $p$ of $\gamma_*\Pbb$. Then the estimate~\eqref{eq:gaponFpd} follows immediately from the spectral gap (Theorem~\ref{thm:SGV}) and the fact (Lemma~\ref{lm:largeOrbits}) that the $\Gamma$-orbit of $x$ has cardinality at least $p$.
\end{proof}
\begin{newchanges}
\section*{Acknowledgement}
We are grateful to the anonymous referee for their numerous suggestions which greatly improved the quality of this paper.
\end{newchanges}

\section*{Funding}
This work was supported by ERC 2020 grant HomDyn (grant no.~833423)

\bibliographystyle{abbrv} 
\bibliography{ref}

\end{document}